\documentclass[reqno,twoside,12pt,draft]{amsart}

\hoffset - 1.5 cm

\DeclareMathSizes{10}{18}{12}{8}   
\usepackage[T1]{fontenc}
\usepackage[final]{graphicx}
\usepackage{epstopdf}
\usepackage{wrapfig}
\usepackage[labelfont=scriptsize,labelfont=bf,textfont=scriptsize]{caption}
\usepackage[labelfont=scriptsize,textfont=scriptsize,justification=raggedright,labelformat=parens]{subcaption}
\usepackage{pict2e}
\usepackage{amstext}
\usepackage{pstricks,pst-node,pst-text,pst-3d}
\usepackage{amsmath}
\usepackage{amsthm}
\usepackage{amssymb}
\usepackage{todonotes}
\usepackage{lipsum}
\reversemarginpar

\newtheorem{Theorem}{Theorem}[section]
\newtheorem{Fact}{Fact}[section]
\newtheorem{Lemma}{Lemma}[section]
\newtheorem{Proposition}{Proposition}[section]
\newtheorem{Corollary}{Corollary}[section]
\theoremstyle{definition}
\newtheorem{Definition}{Definition}[section]
\newtheorem{Example}{Example}[section]
\newtheorem{Remark}{Remark}[section]
\newtheorem{Hipothesis}{Hipothesis}[section]

\newcommand{\ba}{\begin{array}}
\newcommand{\bc}{\begin{center}}
\newcommand{\bd}{\begin{description}}
\newcommand{\bdm}{\begin{displaymath}}
\newcommand{\be}{\begin{enumerate}}
\newcommand{\beq}{\begin{equation}}
\newcommand{\bdf}{\begin{Definition}}
\newcommand{\bex}{\begin{Example}}
\newcommand{\bft}{\begin{Fact}}
\newcommand{\bl}{\begin{Lemma}}
\newcommand{\bp}{\begin{Proposition}}
\newcommand{\br}{\begin{Remark}}
\newcommand{\bt}{\begin{Theorem}}
\newcommand{\bco}{\begin{Corollary}}
\newcommand{\bh}{\begin{Hipothesis}}
\newcommand{\ea}{\end{array}}
\newcommand{\ec}{\end{center}}
\newcommand{\ed}{\end{description}}
\newcommand{\edm}{\end{displaymath}}
\newcommand{\ee}{\end{enumerate}}
\newcommand{\eeq}{\end{equation}}
\newcommand{\edf}{\end{Definition}}
\newcommand{\eex}{\end{Example}}
\newcommand{\eft}{\end{Fact}}
\newcommand{\el}{\end{Lemma}}
\newcommand{\ep}{\end{Proposition}}
\newcommand{\er}{\end{Remark}}
\newcommand{\et}{\end{Theorem}}
\newcommand{\eco}{\end{Corollary}}
\newcommand{\eh}{\end{Hipothesis}}

\newcommand{\numsec}{\setcounter{Theorem}{0}\setcounter{Definition}{0}
\setcounter{Remark}{0} \setcounter{Lemma}{0} \setcounter{Fact}{0}
\setcounter{Proposition}{0} \setcounter{Corollary}{0}
\setcounter{Example}{0} \setcounter{equation}{0}
\setcounter{Property}{0}\renewcommand\theequation{\arabic{section}.\arabic{equation}}
\renewcommand\theTheorem{\arabic{section}.\arabic{Theorem}}
\renewcommand\theDefinition{\arabic{section}.\arabic{Definition}}
\renewcommand\theRemark{\arabic{section}.\arabic{Remark}}
\renewcommand\theLemma{\arabic{section}.\arabic{Lemma}}
\renewcommand\theFact{\arabic{section}.\arabic{Fact}}
\renewcommand\theProposition{\arabic{section}.\arabic{Proposition}}
\renewcommand\theCorollary{\arabic{section}.\arabic{Corollary}}
\renewcommand\theExample{\arabic{section}.\arabic{Example}}
\renewcommand\theProperty{\arabic{section}.\arabic{Property}}}

\numberwithin{equation}{section} \errorcontextlines=0

\newcommand{\dg}{\nabla_G\text{-}\mathrm{deg}}

\usepackage{setspace}

\begin{document}

\title[Global bifurcation]{Bifurcations of critical orbits of invariant potentials with applications to bifurcations of central configurations of the N-body problem}
\author{Marta Kowalczyk}
\address{Faculty of Mathematics and Computer Science\\
Nicolaus Copernicus University in Toru\'n \\
PL-87-100 Toru\'{n} \\ ul. Chopina 12\slash 8 \\
Poland}
\thanks{Partially supported by the National Science Center,  Poland,  under grant    DEC-2012/05/B/ST1/02165}
\email{martusia@mat.umk.pl}

\date{\today}

\maketitle

\begin{abstract}
In this article we study topological bifurcations of critical orbits of equivariant gradient equations. We give necessary and sufficient conditions for the existence of global bifurcations of solutions of these equations. Moreover, we apply these abstract results to the study of bifurcations of new families of planar central configurations of the N-body problem.
\end{abstract}

\numsec
\section{Introduction}
\noindent
The most important questions of celestial mechanics are the ones concerning central configurations. They have a long history dating back to the 18th century and have been studied by various mathematicians, including Euler and Lagrange. 
The first known central configurations were the three collinear classes discovered in 1767 by Euler (see \cite{[EULER]}) and the two classes with masses located at the vertices of an equilateral triangle found in 1772 by Lagrange (see \cite{[LAGRANGE]}). Later it was proved by Wintner that the $3$ collinear and $2$ triangular configurations are the only possible central configurations in the $3$-body problem (see \cite{[WINTNER]}).
In general, the question about the finiteness of number of central configurations has been included in the famous list of eighteen problems of the 21st century by Smale (\cite{[SMALE]}) and it is the following:
\begin{quote}
 ``Given positive real numbers $m_1, \ldots, m_n$ as the masses in
  the n-body problem of celestial mechanics, is the number of relative equilibria finite?''
\end{quote}
As a matter of fact, this open question was already formulated by Wintner in his book (\cite{[WINTNER]}) in 1941. In the $4$- and $5$-body problem the answer to Smale's question (\cite{[SMALE]}) has been given by Hampton and Moeckel (\cite{[HAMPTONMOECKEL]}), and by Kaloshin and Albouy (\cite{[KALOSHINALBOUY]}), respectively. 
Although as for now the question remains unanswered for $6$ and more bodies, there were given some estimations for the number of central configurations. For instance, Merkel in \cite{[MERKEL]} and Pacella in \cite{[PACELLA]}, using Morse theory, determined a lower bound for the number of spatial non-collinear central configurations (for a lower bound of the number of planar non-collinear central configurations see \cite{[PALMORE]}).
For a deeper discussion of central configurations we refer the reader to \cite{[MOECKEL]}.

Classifying central configurations is a very difficult problem, therefore an answer is sought even with some simplifications, that is imposing restrictions on the masses or considering only highly symmetrical solutions. 
One of the many people to go in this direction was Palmore, who considered for $n=4, 5$ the configuration of $n-1$ bodies of mass $1$ at the vertices of a regular $(n-1)$-gon and $n$-th body of arbitrary mass $m$ at its centroid (\cite{[PALMORE1973]}). He showed that for $n=4$ and for $n=5$ there exist unique values of the mass parameter $m$ for which this central configuration becomes degenerate. 
Meyer and Schmidt in 1988 reproduced these results and additionally showed that a bifurcation occurs (\cite{[MEYERANDSCHMIDT8*]}).
Since then a lot more central configurations have been discovered in spirit of these mentioned restrictions.
Examples of central configurations which are considered in our paper are the following:
\begin{itemize}
\item a planar configuration of two nested squares found in $2013$ by A. C. Fernandes, L. F. Mello and M. M. da Silva in \cite{[FERNANDES]}, 
\item a planar rosette configuration of $13$ bodies studied in $2006$ by J. Lei and M. Santoprete in \cite{[LEIANDSANTOPRETE]}, and in $2004$ by M. Sekiguchi in \cite{[SEKIGUCHI]}.  
\end{itemize}

There are numerous reasons for studying central configurations. 
There exists an important class of solutions of the $n$-body problem - the so-called homographic solutions (\cite{[WINTNER]}) - which is directly linked to the concept of central configurations. A solution is
called homographic if at every time the configuration of the bodies 
remains similar. As a matter of fact, this can only happen for central configurations. The Euler and Lagrange solutions, which we have already mentioned, are the most famous examples. Thus central configurations generate the unique explicit solutions of the $N$-body problem known until now. 
In particular, planar central configurations also give rise to families of periodic solutions. As a special case, the homographic solutions for which the configuration just rotates rigidly at constant angular speed are known as relative equilibria.
Moreover, central configurations are important for instance in analysing total collision orbits, namely a certain change of coordinates in the neighbourhood of a total collision point turns the colliding particles into a central configuration in the limit as $t$ approaches the collision time. This phenomenon was shown initially for the complete collapse of all bodies at the centre of mass by Wintner in \cite{[WINTNER]}
and was later generalised by Saari in the study of the parabolic escapes of the particles. The same holds for expanding subsystems, which also asymptotically tend to a central configuration (readers interested in these two phenomena should see \cite{[SAARI]} for details and more references). These examples demonstrate that knowledge of central configurations gives important insight to the dynamics near a total collision and asymptotic behaviour of the universe. 
Another aspect motivating the search for these special solutions is the fact that the hypersurfaces of constant energy and angular momentum change their topology exactly at the energy level sets which contain central configurations (\cite{[SMALEII]}).

Central configurations are invariant under homotheties and rotations - the first claim comes from the homogeneity of the potential while the second one stems from the fact that the potential depends only on the mutual distances between the particles and not on its positions.
Additionally, in the planar case the action of the symmetry group, in this case $SO(2),$ is free, so the quotient space $\Omega/SO(2)$ of the action of the group $SO(2)$ on the configuration space $\Omega$ is a manifold.
Because of this a natural thing to consider is a new potential defined on the quotient space $\Omega/SO(2),$ which has been widely used in the literature
(\cite{[LEIANDSANTOPRETE], [MEYER1987], [MEYERANDSCHMIDT8*], [MEYERANDSCHMIDT1988]}). In this context central configurations have been viewed not as critical orbits, but as critical points of a ``quotient'' potential.
On the other hand, in the spatial case the action of the symmetry group $SO(3)$ is not free, therefore the quotient space $\Omega/SO(3)$ of the action of the group $SO(3)$ on the configuration space $\Omega$ is not a manifold (there are two isotropy groups appearing in the configuration space), so the use of ordinary theory is not justifiable (without any differential structure). Thus in this case we cannot use the approach from the planar one. 
For example, the equivariant Morse theory has been presented to the study of spatial central configurations in \cite{[PACELLA]}. In general, a different approach which does not employ the concept of quotient space has also been successfully used (\cite{[GARCIAIZE], [MACIEJEWSKIRYBICKI], [CHAVELARYBICKI]}).
Apart from the symmetries, Newton's equations are of gradient nature, therefore it seems reasonable to apply methods which benefit from this.

In this paper we use the second approach to study classes of planar central configurations, which are treated as $SO(2)$-orbits of critical points of $SO(2)$-invariant potential.      
Among the machinery which is designed to work in the equivariant context making use both of the symmetry and the gradient structure appearing in the problem and which we employ in our paper are equivariant versions of classical topological invariants such as the degree for equivariant gradient maps (see \cite{[GEBA]} and \cite{[RYBICKI]}) and the equivariant Conley index (see \cite{[BARTSCH]}, \cite{[FLOER]} and \cite{[GEBA]}). 
This approach using equivariant theory has for instance been used in \cite{[MACIEJEWSKIRYBICKI]} and \cite{[CHAVELARYBICKI]}.
For general treatment of the degree for equivariant maps we refer the reader to \cite{[BALANOVKRAWCEWICZSTEINLEIN]} and \cite{[BALANOVKRAWCEWICZRYBICKISTEINLEIN]}. 
In this article we apply those topological tools in equivariant bifurcation theory to produce new classes of planar central configurations from known families by obtaining global bifurcations of families of planar central configurations under a change of the degree for equivariant gradient maps. The same result has been reached for local bifurcations under a change of the equivariant Conley Index (\cite{[SMOLLER]}).
In the planar case because we have exactly one orbit type on the configuration space, the results obtained by using the equivariant Conley Index or the degree for equivariant gradient maps are equivalent to those which one could get considering the quotient space and applying ordinary Conley Index or the Brouwer degree. However, in the spatial case the approach via the quotient space is not possible. In general, the obtained theoretical results are stronger than these we need to study of the planar $N$-body problem.

The aim of this paper is twofold. On the one hand, we give new results about equivariant bifurcation theory (see Theorems \ref{twierbif1}, \ref{twchangedegree} and \ref{twierbif2}).
On the other hand, we apply those theoretical results to the study of planar central configurations.
In the next article we shall present an application of our abstract results to the spatial $N$-body problem and give necessary and sufficient conditions for the existence of local and global bifurcations of spatial central configurations. Thus the main goal of this article is to create a homogeneous theory which can be used for both planar and spatial $N$-body problems. In the planar case, as we mentioned before, we can consider the orbit space of the action of the group $SO(2)$ on the configuration space and then use both ordinary Morse theory or the Brouwer degree and the equivariant approach. However, in the spatial case we cannot work on the quotient space because the action of the group $SO(3)$ on the configuration space is not free.

Notice that we can consider the possibility of working on subsets of the full configuration space which are invariant for the gradient flow and seek bifurcations in these subsets, but it is possible that we can get no bifurcations while there are ones in the full configuration space. This phenomenon happened for instance in \cite{[MEYERANDSCHMIDT8*]}, that is bifurcations of less symmetrical families from highly symmetrical families of central configurations occur. Meyer and Schmidt treated the highly symmetrical family of equilateral triangle with fourth body at the centroid and proved that another families which are less symmetrical (that is the families of isosceles triangles with fourth body near the centroid and on the line of symmetry of the triangle) bifurcate from this family. Similarly, these authors taking the family of square with fifth body at the centroid proved bifurcations of families of kites and isosceles trapezoids (for more examples of this phenomenon see \cite{[MEYERANDSCHMIDT1988]}). In our case of the family of rosette 
configuration of $13$ bodies, likewise, we have the existence of bifurcations of central configurations (not necessarily configurations of rosette type) from the rosette family (see Theorem \ref{thmrosette}) and it transpires that central configurations which bifurcate are not of rosette type. However, we do not know the shape of the bifurcating families. One can compute that there are no values of the mass parameter for which the rosette family is degenerate. Therefore the obtained results are stronger than these which could be get by taking into account the invariant subsets of the configuration space. In the case of the family of two nested squares we also get stronger results not considering invariant subsets.

We stress the fact that the algebraic structure of the potential is not important from the point of view of central configurations, the only things that matter are the rotation and scaling invariance. This makes the methods that we use applicable to mathematical models of physical problems involving interactions between multiple bodies, whose behaviour is not necessarily governed by Newton's gravity laws. Such situation arises for instance in molecular dynamics, where intermolecular
relations are modelled using a broad spectrum of potentials. The famous Lennard-Jones potential, which approximately describes the behaviour of two neutral atoms or molecules, was studied in this context in \cite{[CORBERALLIBRE1]} and \cite{[CORBERALLIBRE2]}.

This paper is organised as follows: Section \ref{preliminaries} contains an introduction to the equivariant setting along with the definition and properties of a topological degree suitable in this context and lemmas used in the proofs of our theorems. In Section \ref{bifurcations} we investigate the $G$-orbits of solutions of the equation
\begin{equation}\label{eq0}
    \nabla _v \varphi (v,\rho )=0 
\end{equation}
under the action of $G,$ strictly speaking we formulate abstract theorems giving necessary and sufficient conditions for the existence of bifurcations in the vicinity of a given family of solutions. 
In Theorem \ref{wkkoniecznybif} we state that only a degenerate critical $G$-orbit of solutions of equation $(\ref{eq0})$ can be a $G$-orbit of local bifurcation. 
The sufficient condition is given by Theorem \ref{twierbif1}, its assumptions, which include a change of the degree for $G$-equivariant gradient maps computed at critical $G$-orbits of $\varphi$, imply the existence of a global bifurcation. We also provide simple conditions in Theorem \ref{twchangedegree} and Remark \ref{twbif1remark1}, which allow us to verify inequality of degrees. Moreover, we prove a theorem in the spirit of Rabinowitz's famous alternative, see Theorem \ref{twierbif2}. In Section \ref{applications} we apply our abstract results to the planar $N$-body problem, namely we consider known families of central configurations - the two squares family \eqref{family_two_squares} and the rosette family \eqref{family_rosette} - and prove the existence of bifurcations of new classes of central configurations from these families. We stress that our method is applicable not only in this specific example, but is rather a general framework, which is additionally 
easily implementable on any computer algebra system, which allows symbolic computations.
Notice that in this paper we consider topological bifurcations (local and global ones, see Definitions \ref{biflok} and \ref{bifglob}) rather than bifurcations in the sense of a change in the number of solutions. In other words, in the case of local one, we consider known family of solutions of equation \eqref{eq0} (so-called trivial solutions) and seek non-trivial ones nearby. 
In the case of global one, we seek connected sets of non-trivial solutions nearby the trivial ones which additionally satisfy Rabinowitz type alternative (see condition \eqref{Rabinowitzalternative}).  
Those two definitions, the bifurcation in the sense of a change in the number of solutions and the topological bifurcation, are independent, that is the first one can occur while the second one does not occur and inversely. Throughout this paper we will write about topological bifurcations briefly bifurcations. 

\numsec
\section{Preliminaries}
\label{preliminaries}
\numsec
In this section we review some classical
facts on equivariant topology. The material is fairly standard and well known (for details see for example \cite{[DIECK2]} and \cite{[KAWAKUBO]}).
Throughout this paper we will remain in the real setting, that is all vector spaces will be considered over $\mathbb{R}$ and all matrices (treated as elements of Lie groups) will have real entries. 
Also, from now on $G$ stands for a compact Lie group and $\mathbb{V}=(\mathbb{R}^n,\varsigma)$ is a finite-dimensional, real, orthogonal $G$-representation, that is a pair consisting of the Euclidean $n$-dimensional vector space and a continuous homomorphism 
$\varsigma: G\rightarrow O(n),$ 
where $O(n)$ denotes the group of orthogonal matrices. By $v\in\mathbb{V}$ we mean $v\in\mathbb{R}^n.$ 
In addition to that by $SO(n)$ we will understand the group of special orthogonal matrices.
The linear action of $G$ on $\mathbb{R}^n$ is given by $\xi :G\times\mathbb{R}^n\rightarrow\mathbb{R}^n,$ where 
$\xi(g,v)=\varsigma(g)v.$
We write $gv$ for short. 
Moreover, $\mathbb{R}$ denotes the one-dimensional, trivial representation of the group $G,$ that is $\mathbb{R}=(\mathbb{R},\mathbf{1}),$ where $\mathbf{1}(g)=1$ for any $g\in G.$ 
Then the linear action of $G$ on $\mathbb{R}^n\times\mathbb{R}$ is given by the Cartesian product of $\varsigma$ and $\mathbf{1}$
\[G\times(\mathbb{R}^n\times\mathbb{R})\ni (g,(v,\rho))\mapsto (gv,\rho)\in \mathbb{R}^n\times\mathbb{R}.\] 
A subset $\Omega\subset\mathbb{V}$ is called $G$-invariant if for any $g\in G$ and $v\in\Omega$ we have $gv\in\Omega.$ 
Let $\overline{sub}(G)$ be the set of closed subgroups of $G.$ 
Two subgroups $H, K \in \overline{sub}(G)$ are called \emph{conjugate} if there exists $g \in G$ such that $H=g^{-1}Kg.$ Conjugacy is an equivalence relation, the conjugacy class of $H$ is denoted by $(H)$ and it is called the orbit type of $H.$ Let $\overline{sub}[G]$ be the set of conjugacy classes of closed subgroups of $G.$ Similarly, $H$ is \emph{subconjugate} to $K$ if $H$ is conjugate to a subgroup of $K,$ denote by $(H)\leq (K).$ Subconjugacy defines a partial order in $\overline{sub}[G].$ Additionally, we write $(H)<(K)$ if $(H)\leq (K)$ and $(H)\neq (K).$

If $v\in\mathbb{V}$ then 
$G_v=\{g\in G:gv=v\}\in\overline{sub}(G)$
is the isotropy group of $v.$ 
For each $v\in\mathbb{V}$ the set 
$G(v)=\{gv:g\in G\}\subset\mathbb{V}$
is called the $G$-orbit through $v.$ 
The isotropy groups of points on the same $G$-orbit are conjugate subgroups of $G.$ We will also denote by $\mathbb{V}^G$ the set of fixed points of the 
$G$-action on $\mathbb{V},$ that is 
$\mathbb{V}^G=\{v\in\mathbb{V}: G_v=G\}=\{v\in\mathbb{V}:gv=v\ \forall g\in G\}.$
Moreover, for given $H\in\overline{sub}(G)$ and an open, $G$-invariant subset $\Omega\subset\mathbb{V}$ we have
$\Omega_{<(H)}=\{v\in\Omega: (G_v)<(H)\}$ and $\Omega_{(H)}=\{v\in\Omega: (G_v)=(H)\}.$
Let $H\in\overline{sub}(G)$ and $\mathbb{W}$ be an orthogonal $H$-representation. Now define an action of $H$ on the product $G\times\mathbb{W}$ by the formula 
$(h,(g,w))\mapsto (gh^{-1},hw)$
and let $G\times_H\mathbb{W}$ denote the space of $H$-orbits of this action.

Fix $k,l\in \mathbb{N}\cup\{\infty\}$ and an open, $G$-invariant subset $\Omega\subset\mathbb{V}.$ A map $\varphi:\Omega\times\mathbb{R}\rightarrow\mathbb{R}$ of class $C^{k}$ is said to be a \emph{$G$-invariant $C^k$-map} if 
$\varphi(gv,\rho)=\varphi(v,\rho)$
for any $g\in G$ and $(v,\rho)\in\Omega\times\mathbb{R}.$ 
The set of $G$-invariant $C^{k}$-maps is denoted by $C^k_G(\Omega\times\mathbb{R},\mathbb{R}).$
A map $\psi:\Omega\times\mathbb{R}\rightarrow\mathbb{V}$ of class $C^{l}$ is said to be a \emph{$G$-equivariant $C^l$-map} if 
$\psi(gv,\rho)=g\psi(v,\rho)$
for any $g\in G$ and $(v,\rho)\in\Omega\times\mathbb{R}.$ 
The set of $G$-equivariant $C^{l}$-maps is denoted by $C^l_G(\Omega\times\mathbb{R},\mathbb{V}).$
For any $G$-invariant $C^k$-map $\varphi$ the gradient of $\varphi$ with respect to the
first coordinate denoted by $\nabla_v\varphi$ is $G$-equivariant $C^{k-1}$-map. 

Let $\Omega\subset\mathbb{V}$ be open and $G$-invariant and fix $v_0\in\Omega,$ where by $H$ we denote $G_{v_0}.$
Notice that $\Omega_{(H)}$ is a $G$-invariant submanifold of $\Omega$ and $G(v_0)$ is a $G$-invariant submanifold of $\Omega_{(H)},$ therefore
we obtain the following orthogonal direct sum decomposition 
\begin{equation}\label{decomp1}
T_{v_0}\mathbb{V}=T_{v_0}\Omega=(T_{v_0}\Omega_{(H)})\oplus(T_{v_0}\Omega_{(H)})^{\bot}=
(T_{v_0}G(v_0))\oplus(T_{v_0}\Omega_{(H)}\ominus T_{v_0}G(v_0))\oplus(T_{v_0}\Omega_{(H)})^{\bot},
\end{equation}
where $T_{v_0}\mathbb{V}$ is the tangent space to $\mathbb{V}$ at $v_0.$ 
On the other hand, we have 
\begin{equation}\label{decomp2}
T_{v_0}\mathbb{V}=T_{v_0}\Omega=(T_{v_0}G(v_0))\oplus\mathbb{W}^H\oplus( \mathbb{W}^H)^{\bot},
\end{equation}
where $\mathbb{W}=(T_{v_0}G(v_0))^{\bot}$ is an orthogonal $H$-representation.
Because $(G\times_HB^{\varepsilon}_{v_0}(\mathbb{W}))_{(H)}=G/H\times B^{\varepsilon}_{v_0}(\mathbb{W}^H),$ we get that
$\mathbb{W}^H\subset T_{v_0}\Omega_{(H)},$ where $B^{\varepsilon}_{v_0}(\mathbb{W})$ denotes the $\varepsilon$-ball centred at $v_0$ in $\mathbb{W}.$
Now assume that $\varphi \in C^2_G(\Omega , \mathbb{R})$ and $\ v_0\in (\nabla \varphi)^{-1} (0).$
The following lemma gives the decomposition of the Hessian $\nabla^2\varphi$ of the potential $\varphi$ with respect to the sum decompositions given by the formulas \eqref{decomp1} or \eqref{decomp2} and has been proved in \cite{[GEBA]}.
\begin{Lemma} \label{hesjanpostac} 
Under the above assumptions, the Hessian
\begin{equation*} \left.\begin{array}{lccccc}
    & & T_{v_0}G(v_0) &  & & T_{v_0}G(v_0)  \\
 & &\oplus & & & \oplus \\
\nabla ^2 \varphi (v_0) \  : & T_{v_0}\mathbb{V}= & \mathbb{W}^H & \rightarrow & T_{v_0}\mathbb{V}= & \mathbb{W}^H \\
 & & \oplus & & & \oplus \\
& & ( \mathbb{W}^H)^{\bot} & & & (\mathbb{W}^H)^{\bot}
\end{array}\right. \end{equation*} is of the form
\begin{equation}\label{hesjanpostacwzor} \nabla ^2 \varphi (v_0)=
\left[ \begin{array}{lcc} 0 & 0 & 0 \\ 0 & B(v_0) & 0 \\ 0 & 0 &
C(v_0)
\end{array}\right]   .
\end{equation}
\end{Lemma}
According to Lemma \ref{hesjanpostac}, we have $\dim \ker \nabla^2\varphi(v_0)\geq\dim G(v_0).$ 
We will thus call a critical $G$-orbit $G(v_0)$ \emph{degenerate} if the strict inequality holds and \emph{non-degenerate} if\break 
$\dim \ker \nabla^2\varphi(v_0)=\dim G(v_0)$ (which is equivalent to $B(v_0)$ and $C(v_0)$ being non-singular). Note that non-degenerate critical $G$-orbits are isolated, that is there exists a $G$-invariant neighbourhood $\Theta\subset\mathbb{V}$ of $G$-orbit $G(v_0)$ such that $(\nabla\varphi)^{-1}(0)\cap \Theta=G(v_0).$ Moreover, a non-degenerate critical $G$-orbit $G(v_0)$ is called \emph{special} if $m^-(C(v_0))=0,$ where by the symbol $m^-(C(v_0))$ we denote the Morse index, that is the number of negative eigenvalues (counting multiplicities) of the matrix $C(v_0).$ 

Regarding planar central configurations, there is exactly one orbit type $(H)$ in the configuration space $\Omega$, that is $\Omega=\Omega_{(H)}$ ($(H)=(\{e\})$), so to illustrate this situation we give the following remark and example:
\begin{Remark}\label{hesjanpostaccccase}
Under the assumptions of Lemma  \ref{hesjanpostac} with $\Omega=\Omega_{(H)},$ the Hessian
\begin{equation*} \left.\begin{array}{lccc}
     & T_{v_0}G(v_0) &  &   T_{v_0}G(v_0) \\
\nabla ^2 \varphi (v_0) \  :  & \oplus & \rightarrow & \oplus \\
&  T_{v_0}\Omega_{(H)}\ominus T_{v_0}G(v_0) & &  T_{v_0}\Omega_{(H)}\ominus T_{v_0}G(v_0)
\end{array}\right. \end{equation*} is of the form
\begin{equation}\label{hesjanpostacwzorcccase} \nabla ^2 \varphi (v_0)=
\left[ \begin{array}{lcc} 0 & 0 \\ 0 & B(v_0)
\end{array}\right]   ,
\end{equation}
where $B(v_0)$ is from the formula \eqref{hesjanpostacwzor}.
\end{Remark}
\begin{Example}\label{hesjanpostacSO(2)}
Let $G=SO(2)$ and $\mathbb{V}=(\mathbb{R}^2,\varsigma),$ where $\varsigma: SO(2)\rightarrow O(2)$ is given by 
$\varsigma(g)=g.$ 
Assume that $\Omega=\mathbb{V}\backslash\{0\},$ then for any $v\in\Omega$ we have $SO(2)_v=\{e\},$ so $\Omega=\Omega_{(\{e\})}.$
Fix $\varphi \in C^2_{SO(2)}(\Omega , \mathbb{R})$ given by $\varphi(v)=\psi(|v|^2),$ where 
$\psi \in C^{\infty}(\mathbb{R} ,\mathbb{R})$ is defined by $\psi(t)=t(t-1)$ and by the symbol $|v|$ we denote the standard norm of an element $v.$ 
Then $\nabla \varphi(v_0)=0$ if and only if $|v_0|^2=\frac{1}{2},$ so we get that for $v_0=(0,\frac{1}{\sqrt{2}})$ the Hessian
\begin{equation*} \left.\begin{array}{lccc}
     & T_{v_0}SO(2)(v_0) &  & T_{v_0}SO(2)(v_0) \\
\nabla ^2 \varphi (v_0) \  :   & \oplus &\rightarrow & \oplus \\
& T_{v_0}\Omega_{(\{e\})}\ominus T_{v_0}SO(2)(v_0) &  & T_{v_0}\Omega_{(\{e\})}\ominus T_{v_0}SO(2)(v_0)  \\
\end{array}\right. \end{equation*} is of the form
\begin{equation*}\nabla ^2 \varphi (v_0)=
\left[ \begin{array}{lcc} 
0 & 0 \\ 0 & 4
\end{array}\right]   .
\end{equation*}
\end{Example}
Note that if there is only one orbit type in $\Omega$ then the quotient space $\Omega/G$ of the action of the group $G$ on the open and $G$-invariant set $\Omega$ is a manifold, so for any $G$-invariant $C^2$-map we can consider a "quotient" potential defined on $\Omega/G.$
In the following lemma we present the decomposition of the Hessian of the "quotient" potential and it transpires that the ordinary Morse index of the "quotient" potential is equal to the Morse index of matrix $B(v_0)$ (see the formula \eqref{hesjanpostacwzor}). 
\begin{Lemma}
Let $\Omega\subset\mathbb{V}$ be open and $G$-invariant. Assume that $\Omega=\Omega_{(H)}$ for some $(H)\in\overline{sub}[G]$ and fix $\varphi \in C^2_G(\Omega , \mathbb{R}).$ Then the potential $\psi:\Omega/G\rightarrow\mathbb{R}$ given by $\psi(G(v_0))=\varphi(v_0)$ is a map of class $C^2$ and $\nabla ^2 \psi (G(v_0)) \  : T_{v_0}\Omega/G  \rightarrow  T_{v_0}\Omega/G$ 
is of the form $\nabla ^2 \psi (G(v_0))=\left[ B(v_0) \right],$ where $B(v_0)$ is from the formula \eqref{hesjanpostacwzor}.
\end{Lemma}

Let $\mathcal{F}_\star(G)$ denote the class of finite, pointed $G$-$CW$-complexes (for the definition of $G$-$CW$-complex see \cite{[DIECK2]}). The $G$-homotopy type of $X\in\mathcal{F}_\star(G)$ is denoted by $[X].$  Let $F$ be a free abelian group generated by $G$-homotopy types of finite, pointed $G$-$CW$-complexes and let $N$ be a subgroup of $F$ generated by elements $[A]-[X]+[X/A],$ where $A,X\in \mathcal{F}_\star(G)$ and $A\subset X.$
Put $U(G)=F/N$ and let $\chi_G(X)$ be the class of an element $[X]\in F$ in $U(G).$
If $X$ is a $G$-$CW$-complex without base point we put $\chi_G(X)=\chi_G(X^+),$ where $X^+=X\cup\{\star\}.$
The ring $U(G)$ is called the Euler ring of $G$ ( for the definition of $U(G)$ see \cite{[DIECK1]} and \cite{[DIECK2]}).
The following theorem can be found in \cite{[DIECK2]}.
\begin{Theorem}
The group $(U(G),+)$ is the free abelian group with basis $\chi_G (G/H^+)$ for $(H)\in\overline{sub}[G].$
Moreover, if $X\in\mathcal{F}_\star(G)$ and $\bigcup _{k=0}^{p}\{(k,(H_{j,k})):\ j=1,\cdots,q(k)\}$ is a type of the cell decomposition of $X,$
then 
\[\chi_G (X)=\sum_{(H)\in\overline{sub}[G]}\left(\sum_{k=0}^p(-1)^{k}\nu(k,(H))\right)\cdot \chi_G (G/H^+),\]
where $\nu(k,(H))$ is the number of cells of dimension $k$ and of orbit type $(H).$
\end{Theorem}

Let $\Omega\subset\mathbb{V}$ be open and $G$-invariant. A function $\varphi\in C^1_G(\mathbb{V},\mathbb{R})$ is said to be \emph{$\Omega$-admissible} if $(\nabla\varphi)^{-1}(0)\cap\partial\Omega=\emptyset.$ 
An $\Omega$-admissible map $\varphi\in C^2_G(\mathbb{V},\mathbb{R})$ is called a \emph{special Morse function} if for any $v_0\in(\nabla \varphi)^{-1} (0)\cap\Omega$ a critical $G$-orbit $G(v_0)$ is special. 

Let $\Omega\subset\mathbb{V}$ be open, $G$-invariant and bounded. For an $\Omega$-admissible function $\varphi\in C^1_G(\mathbb{V},\mathbb{R})$ there is a notion of degree for equivariant gradient maps $\dg(\nabla\varphi,\Omega)\in U(G).$ This invariant has been introduced by G\k{e}ba in \cite{[GEBA]} with coordinates as follows
\[\dg(\nabla\varphi,\Omega)=\sum_{(H)\in\overline{sub}[G]}\dg_{(H)}(\nabla\varphi,\Omega)\cdot \chi_G (G/H^+)\in U(G).\]
To make this article more readable we give a sketch of the definition of the degree. First, it is proved that there exists a special Morse function $\psi\in C^2_G(\mathbb{V},\mathbb{R})$ which is $G$-homotopic to the map $\varphi,$ that is there exists $H \in C^1_G(\mathbb{V}\times[0,1],\mathbb{R})$ being $\Omega$-admissible for each $t\in[0,1]$ such that $H(\cdot,0)=\psi$ and $H(\cdot,1)=\varphi.$ Additionally, because non-degenerate critical $G$-orbits are isolated and $\Omega$ is bounded, we get that $(\nabla\psi)^{-1}(0)\cap\Omega$ consists of a finite number of distinct critical $G$-orbits, that is $(\nabla\psi)^{-1}(0)\cap\Omega=G(v_1)\cup \ldots \cup G(v_l),$ where $G(v_i)\cap G(v_j)=\emptyset$ for $i\neq j.$ Let us remind that for any special critical $G$-orbit $G(v_i)$ we have $m^-(C(v_i))=0,$ so next we allocate $(-1)^{m^-(B(v_i))}$ to each critical $G$-orbit $G(v_i)$ and define the degree for equivariant gradient maps $\dg(\nabla\varphi,\Omega)\in U(G)$ by
\[\dg(\nabla\varphi,\Omega)=\dg(\nabla\psi,\Omega)=\!\!\!\!\!\!\!\!\sum_{(H)\in\overline{sub}[G]}\left(\sum_{(G_{v_i})=(H)}(-1)^{m^-(B(v_i))}\right)\cdot \chi_G (G/H^+)\in U(G).\]
The definition of $\dg(\nabla\varphi,\Omega)$ does not depend on the choice of the special Morse function.
For details of the definition and properties of the degree for equivariant gradient maps see \cite{[GEBA]} and \cite{[RYBICKI]}. 

In the following theorem we formulate some basic properties of this degree.
\begin{Theorem}\label{properties_of_degree}
Fix an $\Omega$-\textit{admissible} function $\varphi\in C^1_G(\mathbb{V},\mathbb{R}),$ where $\Omega$ is an open, bounded and $G$-invariant set. Then the degree $\dg(\nabla\varphi,\Omega)\in U(G)$ has the following properties:
\begin{enumerate}
\item[(1)] \text{[Existence]} if $\dg (\nabla\varphi, \Omega)\neq\mathbf{0}\in U(G)$, then $(\nabla\varphi)^{-1}(0)\cap\Omega\neq\emptyset$,
\item[(2)] \text{[Additivity]} if $\Omega_1,\ \Omega_2$ are open, bounded, $G$-invariant and disjoint sets such that $\Omega=\Omega_1\cup\Omega_2,$ then
\[\dg (\nabla\varphi, \Omega)=\dg (\nabla\varphi, \Omega_1)+\dg (\nabla\varphi, \Omega_2),\]
\item[(3)] \text{[Homotopy invariance]} if $\Phi \in C^1_G(\mathbb{V}\times[0,1],\mathbb{R})$ is $\Omega$-admissible for each $t\in[0,1],$ then
\[\dg(\nabla_v\Phi(\cdot,0), \Omega)=\dg(\nabla_v\Phi(\cdot,1), \Omega),\]
\item[(4)] \text{[Excision]} if $\Omega_1\subset\Omega$ is an open, $G$-invariant set such that $(\nabla\varphi)^{-1}(0)\cap\Omega\subset\Omega_1,$ then 
$\dg (\nabla\varphi, \Omega)=\dg (\nabla\varphi, \Omega_1).$
\end{enumerate}
\end{Theorem}
In the non-equivariant case there are some invariants for continuous maps, for instance the Brouwer degree. This invariant is sometimes too weak to distinguish homotopy classes of two equivariant gradient maps. On the other hand, they can be distinguished by the degree for equivariant gradient maps. In the example below we define two equivariant gradient maps whose the Brouwer degrees are the same while they are distinguished by the degree for $SO(2)$-equivariant gradient maps.
\begin{Example}
Let $G=SO(2)$ and $\mathbb{V}=(\mathbb{R}^2,\varsigma),$ where $\varsigma: SO(2)\rightarrow O(2)$ is given by 
$\varsigma(g)=g.$
Notice that for any $v\in\mathbb{V}$ we have
\begin{equation*}
SO(2)_v=\left\{
\begin{array}{ll}
\{e\} , & v\neq 0\\
SO(2), & v= 0
\end{array}\right.
\end{equation*}
and $B^{1}_0(\mathbb{V})\subset\mathbb{V}$ is an $SO(2)$-invariant set.
Consider two maps $\varphi_\pm\in C^2_{SO(2)}(B^{1}_0(\mathbb{V}),\mathbb{R})$ given by the formulas $\varphi_\pm(v)=\pm\frac{1}{2}|v|^2.$ 
According to Lemma \ref{hesjanpostac}, for $0\in(\nabla\varphi_\pm)^{-1}(0)$ we have 
\begin{equation*}
\nabla ^2 \varphi_+ (0)=
\left[\begin{array}{cc}C_+(0)\end{array}\right]=
\left[\begin{array}{cc}1&0\\ 0&1\end{array}\right]
\text{and}\
\nabla ^2 \varphi_- (0)=
\left[\begin{array}{cc}C_-(0)\end{array}\right]=
\left[\begin{array}{rr}-1&0\\ 0&-1\end{array}\right] .
\end{equation*}
Then the Brouwer degrees of the maps $\nabla\varphi_\pm$ at $0$ in the set $B^{1}_0(\mathbb{V})$ are equal and  
\[\mathrm{deg_B}(\nabla\varphi_+,B^{1}_0(\mathbb{V}),0)=\mathrm{deg_B}(\nabla\varphi_-,B^{1}_0(\mathbb{V}),0)=+1\in\mathbb{Z}.\]
Because $\varphi_+$ is a special Morse function and it has only one critical $SO(2)$-orbit $SO(2)(0),$ we get
\begin{multline*}
\dg (\nabla\varphi_+,B^{1}_0(\mathbb{V}) )=(-1)^{m^-(B_+(0))}\chi_{SO(2)} (SO(2)/SO(2)^+)=\\
=(-1)^0\chi_{SO(2)} (SO(2)/SO(2)^+)=\mathbf{1}\in U(SO(2)). 
\end{multline*}
Next, let $H\in C^2_{SO(2)}(B^{1}_0(\mathbb{V})\times [0,1],\mathbb{R})$ be given by 
$H(v,t)=(1-t\gamma(|v|))\varphi_-(v)+t\gamma(|v|)\varphi_+(v),$ where $\gamma\in C^{\infty}(\mathbb{R},\mathbb{R})$ is such that $\gamma_{|(-\infty,\frac{1}{2}]}\equiv 1,\ \gamma_{|[1,+\infty)}\equiv 0$ and $\gamma_{|(\frac{1}{2},1)}$ is a smooth, decreasing map.
Notice that $H(\cdot,0)=\varphi_-,\ H(\cdot,1)=(1-\gamma(|\cdot|))\varphi_-(\cdot)+\gamma(|\cdot|)\varphi_+(\cdot)$ and $H$ is $B^{1}_0(\mathbb{V})$-admissible for each $t\in[0,1],$ so
the homotopy invariance property of the degree (Theorem \ref{properties_of_degree}.(3)) implies that 
$\dg(\nabla\varphi_- ,B^{1}_0(\mathbb{V}))=\dg(\nabla_vH(\cdot,1) ,B^{1}_0(\mathbb{V})).$
One can compute that the map $H(\cdot,1)$ is a special Morse function and it has exactly two critical $SO(2)$-orbit $SO(2)(0)$ and $SO(2)(v_0),$ where $v_0=(0,y_0)$ and $\frac{1}{2}<y_0<1.$ 
According to Lemma \ref{hesjanpostac}, for $0\in(\nabla_vH(\cdot,1))^{-1}(0)$ and $v_0\in(\nabla_vH(\cdot,1))^{-1}(0)$ we have
\begin{equation*}
\nabla^2_vH(0,1)=\nabla ^2 \varphi_+ (0)\ \text{and}\
\nabla ^2_vH(v_0,1)=
\left[\begin{array}{cc}0 & 0\\ 0 & B(v_0)\end{array}\right] .
\end{equation*}
Then
\begin{multline*}
\dg (\nabla_vH(\cdot,1),B^{1}_0(\mathbb{V}) )=\\ =(-1)^{m^-(B(v_0))}\chi_{SO(2)} (SO(2)/\{e\}^+)+ (-1)^{m^-(B_+(0))}\chi_{SO(2)} (SO(2)/SO(2)^+) =\\=
\pm\chi_{SO(2)} (SO(2)/\{e\}^+)+\chi_{SO(2)} (SO(2)/SO(2)^+)\in U(SO(2)). 
\end{multline*}
\end{Example}
Now we consider equivariant gradient maps whose the Brouwer degree vanishes and the degree for equivariant gradient maps can be nontrivial, see the following example:
\begin{Example}
Let $G=SO(2)$ and $\Omega\subset\mathbb{V}$ be an open, $SO(2)$-invariant and bounded subset of an orthogonal $SO(2)$-representation $\mathbb{V}$ such that $\Omega^{SO(2)}=\emptyset.$ Then for any $\Omega$-admissible map $\varphi\in C^2_G(\Omega,\mathbb{R})$ we have $\mathrm{deg_B}(\nabla\varphi,\Omega,0)=0\in\mathbb{Z}.$
\end{Example}

Below we formulate lemmas which are standard in equivariant topology and will be useful for us in further investigations.
\begin{Lemma}\label{fkoustalonymgihomeo} 
Fix an element $g\in G.$ Then the map $\gamma_g \ : \mathbb{V}\rightarrow \mathbb{V}$ given by the formula $\gamma_g(v)=gv$ is a homeomorphism. In particular, the set $GU=\bigcup_{g\in G}\gamma_g(U)\subset \mathbb{V}$ is open and $G$-invariant for any open subset $U\subset \mathbb{V}.$ 
\end{Lemma}
\begin{Lemma}\label{lmozerach}
Let $\varphi \in C^2_G(\mathbb{V} , \mathbb{R})$ and assume that $U\subset \mathbb{V}.$ Then $(\nabla \varphi)^{-1} (0)\cap GU=\emptyset$ if $(\nabla \varphi)^{-1} (0)\cap U=\emptyset.$
\end{Lemma}
The following theorem can be found in \cite{[BROWN]}.
\begin{Theorem}\label{separationlemma}
 Let $K$ be a compact space and $A,\ B\subset K$ be closed, disjoint sets such that there is no connected set $S\subset K$ such that $S\cap A\neq\emptyset$ and $S\cap B\neq\emptyset.$ Then there exist compact sets $K_A,\ K_B \subset K$ such that $A\subset K_A,\ B\subset K_B,\ K_A\cap K_B=\emptyset$ and $K_A\cup K_B=K.$
\end{Theorem} 
\begin{Lemma} \label{wystepujacetypyorbitowe}
Fix a map $\varphi\in C^{2}_{G}(\mathbb{V},\mathbb{R})$ and $v_0\in\mathbb{V}$ such that $\nabla\varphi(v_0)=0.$ 
Then for any $v\in G B^{\varepsilon}_{v_0}(\mathbb{W})$ we have $(G_v)\leq (G_{v_0}),$ 
where $\mathbb{W}=(T_{v_0}G(v_0))^{\bot}.$
\end{Lemma}
The following lemma can be found in \cite{[MAYER]}.
\begin{Lemma}\label{istnieniespecfunkcjiMorse'a}
Fix a map $\varphi\!\in\! C^{2}_{G}(\mathbb{V},\mathbb{R})$ and assume that $G(v_0)$ is a non-degenerate critical $G$-orbit of $\varphi$ such that
$m^-(C(v_0))\neq 0.$ Then for all open, $G$-invariant neighbourhoods $\Theta \subset\mathbb{V}$ of the
orbit $G(v_0)$ such that $(\nabla\varphi)^{-1}(0)\cap\Theta=G(v_0)$ there exist an open, $G$-invariant neighbourhood $U\subset cl(U)\subset \Theta$ of the orbit $G(v_0),\ \varepsilon >0$ and a map
$\phi\!\in\! C^{2}_{G}(\mathbb{V},\mathbb{R})$ such that
\begin{enumerate}
\item $\varphi (v)=\phi(v)$ for all $v\in \mathbb{V}\backslash U(\varepsilon),$ where $U(\varepsilon)$ denotes an $\varepsilon$-neighbourhood of the set $U,$ that is $\bigcup_{v\in U}B^{\varepsilon}_{v}(\mathbb{V}),$
\item $G(v_0)$ is a special critical $G$-orbit of $\phi,$
\item $((\nabla  \phi )^{-1}(0)\cap (U(\varepsilon)\backslash G(v_0)))_{<(H)}=(\nabla  \phi )^{-1}(0)\cap (U(\varepsilon)\backslash G(v_0)),$ where $(H)=(G_{v_0}).$
\end{enumerate}
\end{Lemma}
\begin{Lemma} \label{postacstopnia}
Fix a map $\varphi\!\in\! C^{2}_{G}(\mathbb{V},\mathbb{R})$ and assume that $G(v_0)$ is a non-degenerate critical $G$-orbit of $\varphi$ such that $m^-(C(v_0))\neq 0$ and $(H)=(G_{v_0}).$ Then there exists an open, $G$-invariant neighbourhood $\Theta$ of the orbit $G(v_0)$ such that $(\nabla\varphi)^{-1}(0)\cap\Theta= G(v_0).$ Moreover, 
\begin{align*}
\dg(\nabla\varphi ,\Theta)& =(-1)^{m^-(B(v_0))} \chi_G(G/H^+)\ +\\ & + \sum_{(H')\in\overline{sub}[G],(H')<(H)}\dg_{(H')}(\nabla\varphi ,\Theta)\cdot \chi_G(G/H'^+).
\end{align*}
\end{Lemma}
\begin{proof}
We can choose an open, $G$-invariant neighbourhood $\Theta$ of the $G$-orbit $G(v_0)$ such that $(\nabla\varphi)^{-1}(0)\cap\Theta= G(v_0).$ Without loss of generality we can assume $(\nabla\varphi)^{-1}(0)\cap\partial\Theta= \emptyset.$
Now, using Lemma \ref{istnieniespecfunkcjiMorse'a}, we define a $\Theta$-admissible homotopy 
$H(v,t)=t\varphi(v)+(1-t)\phi(v).$ 
Next, the homotopy invariance property of the degree (Theorem \ref{properties_of_degree}.(3)) implies that 
\[\dg(\nabla\varphi ,\Theta)=\dg(\nabla_vH(\cdot ,1) ,\Theta)=\dg(\nabla_vH(\cdot ,0) ,\Theta)=\dg(\nabla\phi ,\Theta).\]
In addition, $(\nabla\phi)^{-1}(0)\cap\Theta\subset U\left(\frac{\varepsilon}{2}\right)$ and $(\nabla\phi)^{-1}(0)\cap U\left(\frac{\varepsilon}{4}\right)=\{G(v_0)\},$ just like in Lemma
\ref{istnieniespecfunkcjiMorse'a}. Therefore using excision and additivity properties of the degree (Theorems \ref{properties_of_degree}.(4) and \ref{properties_of_degree}.(2)) and Lemma \ref{wystepujacetypyorbitowe} we can conclude that  
\begin{align*} 
 \nabla_G  \text{-deg}(\nabla \phi ,\Theta) &=\dg(\nabla\phi ,U\left(\varepsilon/ 2\right))= \dg(\nabla\phi ,U\left(\varepsilon/ 4\right)) +\\
 & +\dg(\nabla\phi ,U\left(\varepsilon/ 2\right)\backslash \mathrm{cl}( U\left(\varepsilon/ 4\right)))=(-1)^{m^-(B(v_0))}\chi_G(G/H^+)+\\ 
 &+\sum_{(H')\in \overline{sub}[G],(H')<(H)}\!\!\!\!\!\!\!\!\!\!\dg_{(H')}(\nabla\phi ,U\left(\varepsilon/ 2\right)\backslash \mathrm{cl}( U\left(\varepsilon/ 4\right)))\cdot\chi_G(G/H'^+),
\end{align*}
where by $ \mathrm{cl}( U\left(\varepsilon/ 4\right))$ we denote the closure of $ U\left(\varepsilon/ 4\right).$ 
\end{proof}
 \numsec
\section{Statement of main results}
\label{bifurcations} 
The following assumptions will be needed throughout this section. Let $G$ be a compact Lie group and suppose that $\Omega\subset\mathbb{V}$ is an open, $G$-invariant subset of a
finite-dimensional, real, orthogonal $G$-representation $\mathbb{V}.$ 
Consider a $G$-invariant $C^2$-potential $\varphi:\Omega\times\mathbb{R}\rightarrow\mathbb{R}.$ Since $\nabla_v\varphi$ is $G$-equivariant,  if $(v_0,\rho_0)\in (\nabla_v\varphi)^{-1}(0)$ then $G(v_0)\times\{\rho_0\}\subset(\nabla_v\varphi)^{-1}(0).$ 
A set of the form $G(v_0)\times\{\rho_0\}$ consisting of critical points is called a \emph{critical $G$-orbit of $\varphi.$} 
Using the techniques of equivariant bifurcation theory we will study solutions of the equation
\begin{equation}\label{eq}
    \nabla _v \varphi (v,\rho )=0.
\end{equation}
Additionally, assume there exists a continuous map 
$w:\mathbb{R}\rightarrow\Omega$ such that \[\mathcal{F}=\bigcup\limits_{\rho\in\mathbb{R}}G(w(\rho ))\times\{\rho\}\subset (\nabla_v \varphi )^{-1}(0).\] 
The family $\mathcal{F}$ is called a \emph{trivial family of $G$-orbits} of solutions of \eqref{eq} and for any subset $X\subset\mathbb{R}$ let $\mathcal{F}_{X}$ denote the set of $G$-orbits 
$\bigcup_{\rho\in X }G(w(\rho))\times\{\rho\}\subset\mathcal{F}\subset\Omega\times\mathbb{R}.$

In this section we will formulate the necessary and sufficient conditions for the existence of
bifurcations of nontrivial solutions of (\ref{eq}) from the trivial family $\mathcal{F}.$

Let us introduce a notion of a local bifurcation of $G$-orbits of solutions of (\ref{eq}).
\begin{Definition} \label{biflok}
    Fix parameters $\rho^{\pm}\in\mathbb{R}$ such that $\rho^- < \rho^+.$ A local bifurcation from the segment of $G$-orbits $\mathcal{F}_{[\rho^-,\rho^+]} \subset \mathcal{F}$ of solutions
    of (\ref{eq}) occurs if there exists a $G$-orbit $\mathcal{F}_{\rho_0}\subset\mathcal{F}_{[\rho^-,\rho^+]}$ such that the point $(w(\rho _0) , \rho _0)\in\mathcal{F}_{\rho_0}$ is
    an accumulation point of the set $\{(v,\rho)\in (\Omega\times\mathbb{R}) \backslash\mathcal{F} :\ \nabla _v \varphi (v,\rho )=0\}$ (see Figure \ref{deflocbif}).\\
    We call $\rho _0$ a parameter of local bifurcation and $\mathcal{F}_{\rho_0}$ a $G$-orbit of local bifurcation. We denote by $\mathcal{BIF}$ the set of all parameters of local bifurcation.
\end{Definition}
\begin{figure}[h!]
 \centering
 \setlength{\unitlength}{0.1\textwidth}
  \begin{picture}(5,5)
    \put(0,0){\includegraphics[height=0.4\textwidth]{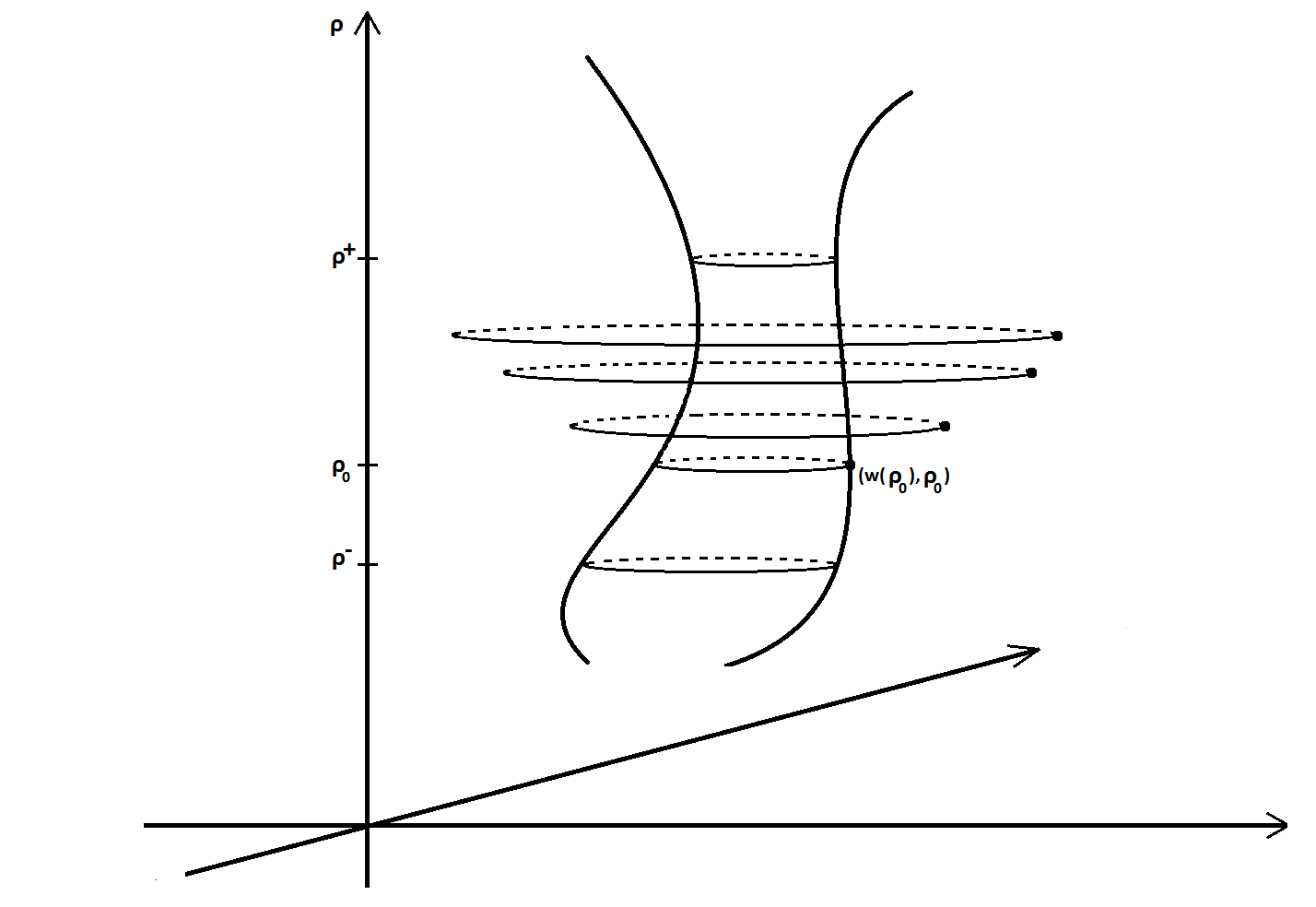}}
    \put(4.17,3.5){$\mathcal{F}$}
    \put(5,0.6){$\mathbb{R}^n$}
  \end{picture}
    \caption{A local bifurcation from the segment of $G$-orbits $\mathcal{F}_{[\rho^-,\rho^+]} \subset \mathcal{F}.$}\label{deflocbif}
\end{figure}
According to Lemma \ref{hesjanpostac}, for each parameter $\rho\in\mathbb{R}$ we have $\dim \ker \nabla_v^2\varphi(w(\rho),\rho)\geq\dim G(w(\rho )).$ 
Now we will formulate the necessary condition for the existence of parameters of local bifurcation.
\begin{Theorem} \label{wkkoniecznybif}
    Under the above assumptions, if, moreover, $\rho _0\in \mathcal{BIF},$ then \[\dim  \ker \nabla ^2_v\varphi (w(\rho_0),\rho_0)>\dim G(w(\rho_0)).\]
\end{Theorem}
\begin{proof}
To prove this theorem we apply two kinds of implicit function theorems - the classical one and the $G$-invariant one.
Suppose, contrary to our claim, that $\rho_0\!\in\!\mathcal{BIF}$ and $\dim \ker \nabla_v^2\varphi (w(\rho _0),\rho_0)\!=\!  \dim G(w(\rho _0)).$ 
Consider two cases: $w(\rho_0)\in \mathbb{V}^G$ and $w(\rho_0)\notin \mathbb{V}^G.$\\
\emph{Case $w(\rho_0)\in \mathbb{V}^G.$} Since $w(\rho_0)\in\mathbb{V}^G,$ we have 
\[G(w(\rho _0))=\{w(\rho_0)\}\ \text{and}\ 0=\dim G(w(\rho _0))=\dim\ker \nabla_v^2\varphi (w(\rho _0),\rho _0),\]
that is $\det \nabla_v^2\varphi (w(\rho_0),\rho _0)\neq 0.$ The map $\varphi$ is of class $C^2$ and $w$ is continuous, so there exists $\varepsilon>0$ such that
$\det\nabla_v^2\varphi (w(\rho),\rho)\neq 0$ for all parameters $\rho\in (\rho_0-\varepsilon,\rho_0+\varepsilon).$ 
The implicit function theorem applied at point $(w(\rho_0),\rho_0)$ yields that there exist open neighbourhoods $U_{w(\rho_0)}\subset\Omega$ and $U_{\rho_0}\subset\mathbb{R}$ of the points $w(\rho_0)\in\Omega$ and $\rho_0\in\mathbb{R},$ respectively, and exactly one continuous map $\psi \ : U_{\rho_0}\rightarrow U_{w(\rho_0)}$ such that
\[\nabla_v\varphi(v,\rho)=0 \mbox{ and } (v,\rho)\in U_{w(\rho_0)}\times U_{\rho_0} \mbox{ if and only if } v=\psi(\rho).\] Therefore $w(\rho)\in \mathbb{V}^G$ for all
parameters $\rho\in U_{\rho_0},\ \psi=w\Big|_{U_{\rho_0}}$ and hence 
\[(\nabla_v\varphi)^{-1}(0)\cap(U_{w(\rho_0)}\times U_{\rho_0}\;\backslash\mathcal{F})=\emptyset ,\]
which contradicts that $\rho_0\in \mathcal{BIF}.$\\
\emph{Case $w(\rho_0)\notin \mathbb{V}^G.$} As a consequence of the $G$-invariant implicit function theorem (see \cite{[DANCER]}) and Remark 4 of \cite{[DANCER]} there exist an open, $G$-invariant neighbourhood $\Theta\subset\Omega$ of the $G$-orbit $G(w(\rho_0))$ and a continuous, $G$-equivariant map
$\psi : G(w(\rho_0))\times (\rho_0-\varepsilon, \rho_0+\varepsilon)\rightarrow \mathbb{V}$ such that if 
\[\mathcal{N}=\{(\psi(v,\rho),\rho):\rho\in(\rho_0-\varepsilon,\rho_0+\varepsilon),v\in G(w(\rho_0))\},\]
then for all $(v,\rho)\in\Theta\times(\rho_0-\varepsilon,\rho_0+\varepsilon)$
\begin{equation}
 \nabla_v\varphi(v,\rho)=0 \ \text{if and only if}\ (v,\rho)\in\mathcal{N}.
\end{equation} 
Hence for all $\rho\in (\rho_0-\varepsilon, \rho_0+\varepsilon)$ if $\nabla_v\varphi(v_1,\rho)=0$ and $\nabla_v\varphi(v_2,\rho)=0,$
where $v_1, \ v_2 \in \Theta,$ then $v_1=\psi(v_1',\rho)$ and $v_2=\psi(v_2',\rho)$ for some $v_1',\ v_2'\in G(w(\rho_0)).$ Thus there exists $g\in G$ such that $v_1'=gv_2'.$
Consequently, $v_1=\psi(v_1',\rho)=\psi(gv_2',\rho)=g\psi(v_2',\rho)=gv_2.$ On the other hand, we have the map $w :\mathbb{R}\rightarrow\Omega$ such that
$\nabla_v\varphi(w(\rho),\rho)=0$ for all parameters $\rho\in \mathbb{R}.$ 
Thus for all $\rho\in(\rho_0-\varepsilon,\rho_0+\varepsilon)$ if $w(\rho)\in\Theta$ then $(w(\rho),\rho)\in\mathcal{N},$ that is $w(\rho)=\psi(gw(\rho_0),\rho)$ for some $g\in G.$ Since $w$ is continuous and $w(\rho_0)\in\Theta,$ there exists small enough $\tilde{\varepsilon}\leq\varepsilon$ such that 
$w(\rho)\in\Theta$ for all $\rho\in(\rho_0-\tilde{\varepsilon},\rho_0+\tilde{\varepsilon}).$ 
Hence $(\nabla_v\varphi)^{-1}(0)\cap(\Theta\times (\rho_0-\tilde{\varepsilon},\rho_0+\tilde{\varepsilon})\backslash\mathcal{F})=\emptyset,$ which again contradicts that $\rho_0\in \mathcal{BIF}.$
\end{proof}
After introducing the local bifurcation of $G$-orbits of solutions of (\ref{eq}) we now turn to the notion of global bifurcation of these solutions. Denote by $C(\rho_0)$ the connected component of the set 
$cl(\{(v,\rho)\!\in\!\! (\Omega\times\mathbb{R}) \backslash\mathcal{F}: \nabla _v \varphi (v,\rho )=0\})\cup\mathcal{F}_{\rho_0}$ containing $\mathcal{F}_{\rho_0}$
and by $C([\rho^-,\rho^+])$ the connected component of the set $cl(\{(v,\rho)\!\in\! (\Omega\times\mathbb{R})\backslash\mathcal{F} : \nabla _v \varphi (v,\rho )=0\})\cup\mathcal{F}_{[\rho^-,\rho^+]}$ containing $\mathcal{F}_{[\rho^-,\rho^+]}.$
\begin{Definition}\label{bifglob}
    Fix parameters $\rho^{\pm}\in\mathbb{R}$ such that $\rho^- < \rho^+.$ A global bifurcation from the segment of $G$-orbits $\mathcal{F}_{[\rho^-,\rho^+]} \subset \mathcal{F}$ of solutions
    of (\ref{eq}) occurs if the component $C([\rho^-,\rho^+])\subset\Omega\times\mathbb{R}$ 
    satisfies the following condition:     
    \begin{equation}\label{Rabinowitzalternative}
     C([\rho^-,\rho^+])\ \text{is not compact or}\  
    (C([\rho^-,\rho^+])\backslash \mathcal{F}_{[\rho^-,\rho^+]})\cap \mathcal{F}\neq \emptyset
    \end{equation}
    (see Figure \ref{defglobbif}).\\
    We call a parameter $\rho_0\in[\rho^-,\rho^+]$ a parameter of global bifurcation if the component $C(\rho_0)\subset\Omega\times\mathbb{R}$ is not compact or 
    $(C(\rho_0)\backslash \mathcal{F}_{[\rho^-,\rho^+]})\cap \mathcal{F}\neq \emptyset.$ A $G$-orbit $\mathcal{F}_{\rho_0}\subset\mathcal{F}_{[\rho^-,\rho^+]}$ is called a $G$-orbit of global bifurcation. We denote by $\mathcal{GLOB}$ the set of all parameters of global bifurcation.
\end{Definition}
\begin{figure}[h!]
 \centering
 \setlength{\unitlength}{0.1\textwidth}
    \begin{picture}(4,4)
  \put(0,0){\includegraphics[height=0.4\textwidth]{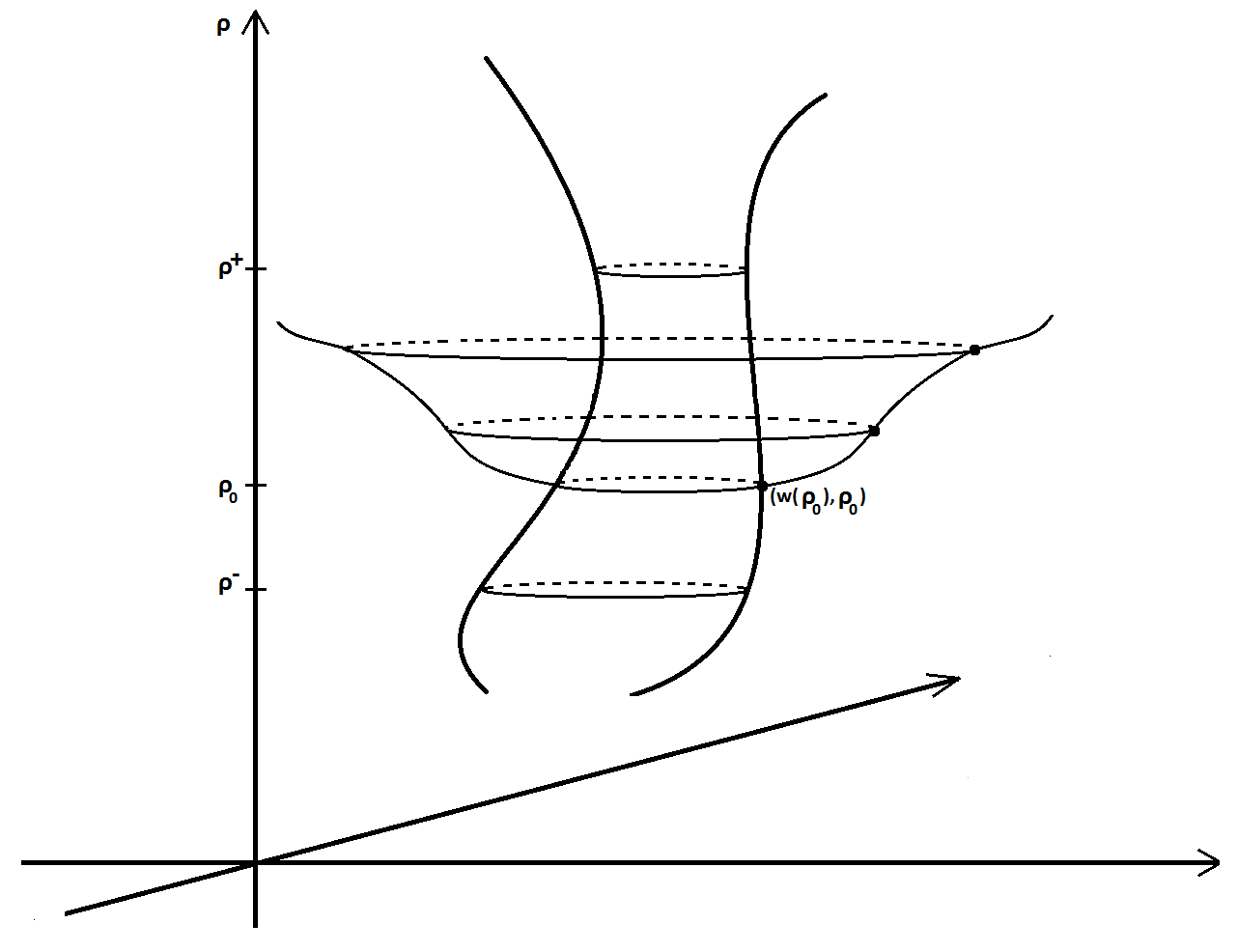}}
  \put(3.6,3.5){$\mathcal{F}$}
  \put(4.5,0.6){\makebox(0,0)[b]{$\mathbb{R}^n$}}
  \end{picture}
   \caption{A global bifurcation from the segment of $G$-orbits $\mathcal{F}_{[\rho^-,\rho^+]} \subset \mathcal{F}.$}\label{defglobbif}
  \end{figure}
The condition \eqref{Rabinowitzalternative} is referred to as the Rabinowitz type alternative (\cite{[RABINOWITZ]}). 
Notice that $\mathcal{GLOB}\subset\mathcal{BIF}.$ Below we define the bifurcation index from the segment $[\rho^-,\rho^+]$ to be an element $BIF_{[\rho^-,\rho^+]}\in U(G).$
\begin{Definition} \label{indbif}
    Fix parameters $\rho^{\pm}\in\mathbb{R}, \ \rho^- < \rho^+$ such that $\rho^{\pm}\notin \mathcal{BIF}$ and fix open, $G$-invariant neighbourhoods $\Theta^{\pm}\subset\Omega$ of
    $G$-orbits $G(w(\rho^{\pm}))$ such that 
    \[(\nabla_v\varphi)^{-1}(0)\cap(\mathrm{cl}(\Theta^{\pm})\times\{\rho^{\pm}\})=\mathcal{F}_{\rho^{\pm}}.\] We define the bifurcation index from the segment $[\rho^-,\rho^+]$ by
    \[BIF_{[\rho^-,\rho^+]}=\dg(\nabla_v\varphi(\cdot , \rho^{+}),\Theta^+)- \dg(\nabla_v\varphi(\cdot , \rho^{-}) ,\Theta^-)\in U(G).\]
\end{Definition}
The definition of $BIF_{[\rho^-,\rho^+]}$ does not depend on the choice of the $G$-invariant neighbourhood of the $G$-orbit. In what follows we formulate a sufficient condition for the existence of global
bifurcation from the segment of solutions of $(\ref{eq})$.
\begin{Theorem}\label{twierbif1}
   Fix parameters $\rho^{\pm}\in\mathbb{R}, \ \rho^- < \rho^+$ such that $\rho^{\pm}\notin \mathcal{BIF}$ and assume that $BIF_{[\rho^-,\rho^+]}\neq \mathbf{0} \in U(G).$ Then a global bifurcation
   from the segment of $G$-orbits $\mathcal{F}_{[\rho^-,\rho^+]} \subset \mathcal{F}$ of solutions of (\ref{eq}) occurs.
\end{Theorem}
\begin{proof}
First, we will prove that there exists a parameter $\rho_0\in (\rho^{-},\rho^{+}) $ such that $\rho_0\in\mathcal{BIF}.$ 
Suppose, contrary to our claim, that $BIF_{[\rho^-,\rho^+]}\neq \mathbf{0}$ and $\mathcal{BIF}\cap [\rho^-, \rho^+]= \emptyset.$ 
Hence there exists an open, bounded $G$-invariant neighbourhood $U\subset\Omega\times\mathbb{R}$ of the set $\mathcal{F}_{[\rho^-,\rho^+]}$ such that 
$(\nabla _v\varphi)^{-1}(0)\cap (\mathrm{cl}(U)\backslash\mathcal{F})=\emptyset .$
In particular, we can pick open, bounded $G$-invariant neighbourhoods $\Theta^{\pm}\subset\Omega$ of $G$-orbits $G(w(\rho^{\pm}))$ such that 
$(\nabla_v\varphi)^{-1}(0)\cap(\mathrm{cl}(\Theta^{\pm})\times\{\rho^{\pm}\})=\mathcal{F}_{\rho^{\pm}}.$ 
The generalised homotopy invariance property of the degree (Theorem \ref{properties_of_degree}.(3)) yields that 
\[\dg(\nabla\varphi(\cdot , \rho^-) ,U\cap(\Omega\times\{\rho^-\}))= \dg(\nabla\varphi(\cdot , \rho^+) ,U\cap(\Omega\times\{\rho^+\})),\] 
while from excision property (Theorem \ref{properties_of_degree}.(4)) we obtain the following equalities: 
\begin{align*}
\dg(\nabla\varphi(\cdot ,\rho^-), U\cap(\Omega\times\{\rho^-\}))=\dg(\nabla\varphi(\cdot ,\rho^-) ,\Theta^-),\\  
\dg(\nabla\varphi(\cdot , \rho^+),U\cap(\Omega\times\{\rho^+\}))= \dg(\nabla\varphi(\cdot , \rho^+) ,\Theta^+),
\end{align*}
a contradiction. 
We have thus proved that there exists a $G$-orbit of local bifurcation $\mathcal{F}_{\rho_0}$ of solutions of (\ref{eq}), or equivalently, the parameter $\rho_0$ is a parameter of local bifurcation, that is $\rho_0\in\mathcal{BIF}.$\\
Secondly, we will prove that a global bifurcation from 
$\mathcal{F}_{[\rho^-,\rho^+]}$ occurs. The proof will be divided into two steps.\\
\emph{Step 1.} We first prove that $C([\rho^-,\rho^+])\neq \mathcal{F}_{[\rho^-,\rho^+]}.$ Suppose the contrary and choose an open, bounded set
$Q\subset\Omega\times\mathbb{R}$ such that $\mathcal{F}_{(\rho^-,\rho^+)}\subset Q$ and 
$\Omega\times((-\infty,\rho^-)\cup(\rho^+,+\infty))\cap Q=\emptyset.$
Additionally, let 
$(\Omega\times\{\rho^{\pm}\})\cap\partial Q=\mathrm{cl}(\Theta^\pm)\times\{\rho^\pm\},$
where $\Theta^{\frac{+}{}}\subset \Omega$ are chosen similarly as above, that is $\Theta^{\frac{+}{}}\subset \Omega$ are open, bounded, $G$-invariant neighbourhoods of the $G$-orbits
$G(w(\rho^{\frac{+}{}}))$ 
such that $(\nabla_v\varphi)^{-1}(0)\cap(\mathrm{cl}(\Theta^{\pm})\times\{\rho^{\pm}\})=\mathcal{F}_{\rho^{\pm}}.$
Observe that $\mathrm{cl}(Q)\cap (\nabla _v \varphi )^{-1}(0)$ is compact and put 
\begin{align*}
 K & =\mathrm{cl}(Q)\cap (\nabla _v \varphi )^{-1}(0),\\
 A & =\mathcal{F}_{[\rho^{-},\rho^{+}]},\\
 B & =\partial Q\cap \mathrm{cl}(\{(v,\rho)\in (\Omega\times\mathbb{R}) \backslash\mathcal{F} :\ \nabla _v \varphi (v,\rho )=0\}).
\end{align*} 
These sets satisfy assumptions of Theorem \ref{separationlemma}, hence we obtain sets $K_A$ and $K_B$ from the conclusion of this theorem. Since these sets are compact and
disjoint, there exists $\eta >0$ such that $\eta$-neighbourhoods $K_A(\eta),\ K_B(\eta)$ of the sets $K_A, \ K_B$ are disjoint. 
Now put 
\[U=G(Q\backslash \mathrm{cl}(K_B(\frac{1}{2}\eta)))\]
and observe that $\partial U \cap (\nabla _v\varphi)^{-1}(0)= \mathcal{F}_{\rho^{\pm}}$ by Lemma \ref{lmozerach}. 
The set $U$ is $G$-invariant and open by Lemma \ref{fkoustalonymgihomeo}. Again, using the generalised homotopy invariance property of the degree (Theorem \ref{properties_of_degree}.(3)), we obtain the equality
\[\dg(\nabla\varphi(\cdot , \rho^-) ,\Theta^-)=\dg(\nabla\varphi(\cdot , \rho^+) ,\Theta^+),\] a contradiction.\\
\emph{Step 2.} Now we prove that either the component $C([\rho^-,\rho^+])$ is not compact or that the sets $C([\rho^-,\rho^+])\backslash \mathcal{F}_{[\rho^-,\rho^+]}$ and $\mathcal{F}$ have nonempty intersection. Suppose the contrary and pick an open, bounded $\varepsilon$-neighbourhood
$(\mathrm{cl}(C([\rho^-,\rho^+])\backslash \mathcal{F}_{[\rho^-,\rho^+]}))(\varepsilon)\subset\Omega\times\mathbb{R}$ of $\mathrm{cl}(C([\rho^-,\rho^+])\backslash \mathcal{F}_{[\rho^-,\rho^+]})$ such that 
$\mathrm{cl}((\mathrm{cl}(C([\rho^-,\rho^+])\backslash \mathcal{F}_{[\rho^-,\rho^+]}))(\varepsilon))\cap \mathcal{F}_{(-\infty,\rho^{-}]\cup[\rho^+,+\infty)}=\emptyset$
for some (small enough) $\varepsilon>0.$ Let $Q\subset\Omega\times\mathbb{R}$ and $\Theta^{\pm}\subset\Omega$ be the sets chosen as in Step 1 satisfying an additional condition: 
$(\Theta^{\frac{+}{}}\times\{\rho^{\frac{+}{}}\})\cap (\mathrm{cl}(C([\rho^-,\rho^+])\backslash \mathcal{F}_{[\rho^-,\rho^+]}))(\varepsilon)=\emptyset.$ 
Now define 
$Q_1=Q\cup (\mathrm{cl}(C([\rho^-,\rho^+])\backslash \mathcal{F}_{[\rho^-,\rho^+]}))(\varepsilon)$
and notice that $\mathrm{cl}(Q_1)\cap (\nabla _v \varphi )^{-1}(0)$ is a compact set. Again put 
\begin{align*} 
K & =\mathrm{cl}(Q_1)\cap (\nabla _v \varphi )^{-1}(0),\\
A & =C([\rho^-,\rho^+]),\\
B & =\partial Q_1\cap \mathrm{cl}(\{(v,\rho)\in (\Omega\times\mathbb{R})\backslash\mathcal{F} :\ \nabla _v \varphi (v,\rho )=0\}).
\end{align*}
The rest follows analogously to Step 1, but with $Q_1$ instead of $Q$. So defined sets satisfy the assumptions of Theorem \ref{separationlemma}, therefore we obtain $K_A$ and $K_B$ in an
identical manner as earlier. Once again they are compact and disjoint, which enables us to pick their disjoint $\eta$-neighbourhoods $K_A(\eta)$ and $K_B(\eta)$. Similarly to Step 1, we now put
\[U=G(Q_1\backslash \mathrm{cl}(K_B(\frac{1}{2}\eta)))\]
and notice that $U$ is open, $G$-invariant and $\partial U \cap (\nabla _v \varphi )^{-1}(0)= \mathcal{F}_{\rho^{\pm}}.$ 
Once again the generalised  homotopy invariance property of the degree (Theorem \ref{properties_of_degree}.(3)) implies equality of degrees, a contradiction.
\end{proof}
A consequence of Theorem \ref{wkkoniecznybif} is that there is no $G$-orbit of any type of bifurcation among non-degenerate $G$-orbits, which is made precise by the following:
\begin{Corollary} \label{twbif1corollary}
    Fix parameters $\rho^{\pm},\rho_0\in\mathbb{R}, \ \rho^- < \rho_0< \rho^+$ and assume that $\mathcal{F}_{\rho_0}$ is the only degenerate $G$-orbit in the segment $[\rho^-,\rho^+].$ 
    Thus, if $BIF_{[\rho^-,\rho^+]}\neq \mathbf{0} \in U(G),$ then there exists exactly one $G$-orbit of local bifurcation in the segment $[\rho^-,\rho^+]$ and this orbit is precisely $\mathcal{F}_{\rho_0}.$
    Moreover, it is also a $G$-orbit of global bifurcation.
\end{Corollary}
What remains is a question how to distinguish the degrees of critical orbits. It transpires that under a very mild assumption this can be done by means of a simple condition.
\begin{Theorem} \label{twchangedegree} 
Fix parameters $\rho^{\pm}\in\mathbb{R}, \ \rho^- < \rho^+$ and assume additionally that the $G$-orbits $G(w(\rho^\pm))$ are non-degenerate and $(G_{w(\rho^+)})=(G_{w(\rho^-)})=(H).$ Thus, if\\ 
$\det B(w(\rho^-))\cdot \det B(w(\rho^+)) < 0,$ then  
\[\dg(\nabla_v\varphi(\cdot , \rho^{-}) ,\Theta^{-})\neq \dg(\nabla_v\varphi(\cdot , \rho^{+}) ,\Theta^{+}).\]
\end{Theorem}
\begin{proof}
According to Lemma \ref{hesjanpostac}, there is a special form of $\nabla^2_v\varphi(w(\rho^{\pm}),\rho^{\pm})$ given by the formula (\ref{hesjanpostacwzor}). Observe that if $m^-(C(w(\rho^\pm)))=0,$
then 
\begin{equation}\label{formula1}
\dg(\nabla_v\varphi(\cdot , \rho^\pm) ,\Theta^\pm)= (-1)^{m^-(B(w(\rho^\pm)))}\chi_G(G/H^+).
\end{equation}
When, on the other hand, $m^-(C(w(\rho^\pm)))\neq 0,$ by Lemma \ref{postacstopnia}, we obtain that 
\begin{align}\label{formula2}
\dg(\nabla_v\varphi(\cdot, \rho^\pm) ,\Theta^\pm) & = (-1)^{m^-(B(w(\rho^\pm)))}\chi_G(G/H^+)\ +\nonumber \\ & + \sum_{(H')\in \overline{sub}[G],(H')<(H)}\dg_{(H')}(\nabla_v\varphi(\cdot,
\rho^\pm) ,\Theta^\pm)\cdot\chi_G(G/H'^+).
\end{align}
Since $\det B(w(\rho^-))\cdot \det B(w(\rho^+)) < 0,$ we get that $(-1)^{m^-(B(w(\rho^-)))}\neq(-1)^{m^-(B(w(\rho^+)))}.$ Using the formulas \eqref{formula1} and \eqref{formula2} we complete the proof.
\end{proof}
\begin{Remark} \label{twbif1remark1} 
In the case of Theorem \ref{twchangedegree} with $(G_{w(\rho^-)})\neq(G_{w(\rho^+)}),$ we get inequality of the degrees $\dg(\nabla_v\varphi(\cdot , \rho^{\pm}) ,\Theta^{\pm}),$ what is the consequence of the formulas (\ref{formula1}) and (\ref{formula2}).   
\end{Remark}
The following theorem is in the spirit of Rabinowitz's theorems regarding the problem of global bifurcation of zeros without symmetries (see \cite{[RABINOWITZ]}).
\begin{Theorem} \label{twierbif2}
Fix parameters $\rho^{\pm}\!\in\!\mathbb{R},\ \rho^{-}<\rho^{+}$ such that $\rho^\pm\notin\mathcal{BIF}$ and $BIF_{[\rho^-,\rho^+]}\neq \mathbf{0} \in U(G).$ 
Then a global bifurcation from the segment of $G$-orbits $\mathcal{F}_{[\rho^-,\rho^+]} \subset \mathcal{F}$ of solutions of (\ref{eq}) occurs, 
that is  the component $C([\rho^-,\rho^+])\subset\Omega\times\mathbb{R}$ is not compact or $(C([\rho^-,\rho^+])\backslash \mathcal{F}_{[\rho^-,\rho^+]})\cap \mathcal{F}\neq \emptyset.$
Moreover, if the component $C([\rho^-,\rho^+])$ is compact and there exist finitely many parameters 
\[\rho^-_0<\rho^+_0<\ldots<\rho^-_l=\rho^-<\rho^+=\rho^+_l<\ldots<\rho^-_n<\rho^+_n\] such that
\begin{enumerate}
\item $C([\rho^-,\rho^+])\cap\mathcal{F}\subset\bigcup_{k=0}^{n}\mathcal{F}_{[\rho^-_k,\rho^+_k]},$
\item $C([\rho^-,\rho^+])\cap\mathcal{F}_{[\rho^-_k,\rho^+_k]}\neq\emptyset$ for $k=0,\ldots,n,$
\item $\mathcal{GLOB}\cap \bigcup_{k=0}^{n}\mathcal{F}_{[\rho^-_k,\rho^+_k]}=\mathrm{cl}(C([\rho^-,\rho^+])\backslash \mathcal{F}_{[\rho^-,\rho^+]})\cap\bigcup_{k=0}^{n}\mathcal{F}_{[\rho^-_k,\rho^+_k]},$
\item $\rho^\pm_k\notin\mathcal{BIF}$ for $k=0,\ldots,n,$
\end{enumerate}
then
\begin{equation}\sum_{k=0}^{n}BIF_{[\rho^-_k,\rho^+_k]}=\mathbf{0}\in U(G).\label{sumastopni}\end{equation}
\end{Theorem}
\begin{proof}
According to Theorem \ref{twierbif1}, we only need to show that the formula (\ref{sumastopni}) is valid. To prove this, fix an open and bounded $\varepsilon$-neighbourhood $(\mathrm{cl}(C([\rho^{-},\rho^{+}])\backslash \mathcal{F}_{[\rho^{-},\rho^{+}]}))(\varepsilon),$ denoted by $W$ for short, of
the set $\mathrm{cl}(C([\rho^{-},\rho^{+}])\backslash \mathcal{F}_{[\rho^{-},\rho^{+}]})$ such that
$\mathrm{cl}(W)\cap\mathcal{F}_{\mathbb{R}\backslash\bigcup_{k=0}^{n}(\rho^{-}_k,\rho^{+}_k)}=\emptyset$
for some (small enough) $\varepsilon >0.$ 
By assumption $(3),$ for any $k\neq l$ we can choose an open, bounded $\tilde{\varepsilon}$-neighbourhood 
$(\mathrm{cl}(C([\rho^{-}_k,\rho^{+}_k])\backslash (\mathcal{F}_{[\rho^{-}_k,\rho^{+}_k]}\cup C([\rho^{-},\rho^{+}]))))(\tilde{\varepsilon}),$ denoted by $W_k$ for short, of $\mathrm{cl}(C([\rho^{-}_k,\rho^{+}_k])\backslash (\mathcal{F}_{[\rho^{-}_k,\rho^{+}_k]}\cup C([\rho^{-},\rho^{+}])))$ such that 
$\mathrm{cl}(W_k)\cap\mathcal{F}_{\mathbb{R}\backslash(\rho^{-}_k,\rho^{+}_k)}=\emptyset$ for some (small enough) $\tilde{\varepsilon} >0.$ 
For any $k\in\{0,\ldots ,n\}$ set an open and bounded set $Q_k\subset\Omega\times\mathbb{R}$ such that
$\mathcal{F}_{(\rho^{-}_k,\rho^{+}_k)}\subset Q_k$ and 
$(\Omega\times((-\infty,\rho^-_k)\cup(\rho^+_k,+\infty)))\cap Q_k=\emptyset.$ 
Additionally, let 
$(\Omega\times\{\rho^{\pm}_k\})\cap\partial Q_k=\mathrm{cl}(\Theta^\pm_k)\times\{\rho^\pm_k\},$
where $\Theta^{\pm}_k\subset \Omega$ are an open, bounded, $G$-invariant neighbourhoods of the $G$-orbits
$G(w(\rho^{\frac{+}{}}_k))$ 
such that $(\nabla_v\varphi)^{-1}(0)\cap(\mathrm{cl}(\Theta^{\pm}_k)\times\{\rho^{\pm}_k\})=\mathcal{F}_{\rho^{\pm}_k}$ and
$(\Theta^{\pm}_k\times\{\rho^{\pm}_k\})\cap W=\emptyset.$
For $k\neq l$ let $\Theta_k^\pm$ satisfy an additional condition  
$(\Theta_k^\pm\times \{\rho_k^\pm\})\cap W_k=\emptyset.$  
Now for $k\neq l$ define $\tilde{Q}_k=Q_k\cup W_k$ and for $k=l$ let $\tilde{Q}_l=Q_l.$ 
Next, define 
$Q= \bigcup_{k=0}^{n} \tilde{Q}_k\cup W$ and observe that $\mathrm{cl}(Q)\cap (\nabla _v \varphi )^{-1}(0)$ is a compact set. Put 
\begin{align*}
 K & =\mathrm{cl}(Q)\cap (\nabla _v \varphi )^{-1}(0),\\
 A & =\bigcup_{k=0}^{n}\mathcal{F}_{[\rho^{-}_k,\rho^{+}_k]}\cup C([\rho^-,\rho^+])\\
 B & =\partial Q\cap \mathrm{cl}(\{(v,\rho)\in (\Omega\times\mathbb{R})\backslash\mathcal{F} :\ \nabla _v \varphi (v,\rho )=0\}).
\end{align*} 
Sets so defined satisfy the assumptions of Theorem \ref{separationlemma}, hence we obtain sets $K_A$ and $K_B$ from the conclusion of this theorem. Since these sets are
compact and disjoint, there exists $\eta >0$ such that $\eta$-neighbourhoods $K_A(\eta),\ K_B(\eta)$ of the sets $K_A, \ K_B$ are disjoint. 
Now put 
$U=G(Q\backslash \mathrm{cl}(K_B(\frac{1}{2}\eta)))$
and observe that $(\nabla_v \varphi )^{-1}(0)\cap\partial U= \bigcup_{k=0}^{n}\mathcal{F}_{\rho^-_k}\cup\mathcal{F}_{\rho^+_k}$ by Lemma
\ref{lmozerach}. The set $U$ is $G$-invariant and open by Lemma \ref{fkoustalonymgihomeo}. Using the generalised homotopy invariance property of the degree (Theorem \ref{properties_of_degree}.(3)) we obtain that \[\dg(\nabla_v\varphi(\cdot , \rho),U\cap(\Omega\times\{\rho\}))\in U(G)\] is well defined and is constant as a function of the 
$\rho\in [\rho_{min}+\varepsilon_0,\rho_{max}-\varepsilon_0],$ where
\begin{align*}
\rho_{min}&=inf\{\rho\in\mathbb{R}: U\cap(\Omega\times\{\rho\})\neq\emptyset\},\\ \rho_{max}&=sup\{\rho\in\mathbb{R}:U\cap(\Omega\times\{\rho\})\neq\emptyset\} 
\end{align*}
and $\varepsilon_0$ is
a positive number satisfying 
$(\nabla_v\varphi)^{-1}(0)\cap (U\cap(\Omega\times\{\rho_{\substack{min \\ max}}\pm\varepsilon_0\}))=\emptyset$
in the case of $\rho_{\text{min}}< \rho^-_0$ and $\rho_{\text{max}}>\rho_n^+.$ In other cases, $\rho\in[\rho^-_0,\rho_{max}-\varepsilon_0]$ or $\rho\in[\rho_{min}+\varepsilon_0, \rho^+_n]$ or $\rho\in[\rho^-_0,\rho^+_n].$ 
Observe that, using the additivity property of the degree (Theorem \ref{properties_of_degree}.(2)), we have 
\begin{multline*}\dg(\nabla_v\varphi(\cdot , \rho^+_k),U\cap(\Omega\times\{\rho^+_k\}))=\\ \underbrace{\dg(\nabla_v\varphi(\cdot , \rho^+_k) ,\Theta^+_k)}_{a^+_k}+ \underbrace{\dg(\nabla_v\varphi(\cdot , \rho^+_k),(U\cap(\Omega\times\{\rho^+_k\}))\backslash\mathrm{cl}(\Theta^+_k))}_{c^+_k},
\end{multline*}
analogously for parameters $\rho^-_k,k=0,\dots,n.$ We assume that $\dg(\nabla_v\varphi(\cdot , \rho),\emptyset)=\mathbf{0}\in U(G).$
Since $a_k^++c_k^+=a_{k+1}^-+c_{k+1}^-,$ we have $a_k^+=a_{k+1}^-,$ hence we conclude that
\begin{equation*}
\sum_{k=0}^{n}a^+_k-a^-_k =\sum_{k=0}^{n}a^+_k-\sum_{k=0}^{n}a^-_k=\sum_{k=0}^{n}a^+_k-\sum_{k=-1}^{n-1}a^-_{k+1}=\sum_{k=0}^{n-1}(a^+_k-a^-_{k+1})+a^+_n -a^-_0=a^+_n-a^-_0.
\end{equation*}
Observe also that $c^+_n=c^-_0=0.$ Then
\begin{multline*}
\sum_{k=0}^{n} BIF_{[\rho^-_k,\rho^+_k]}=
\sum_{k=0}^{n}\dg(\nabla_v\varphi(\cdot , \rho^+_k),\Theta^+_k)-\dg(\nabla_v\varphi(\cdot , \rho^-_k),\Theta^-_k)=\\
=\sum_{k=0}^{n}a_k^+-a_k^-=a^+_n-a^-_0=(a_n^++c_n^+)-(a_0^-+c_0^-)=\mathbf{0}\in U(G),
\end{multline*}
which proves our assertion.
\end{proof}
\begin{Remark}
Observe that in Theorem \ref{twierbif2} if for $k=0,\ldots,n$ the set $\mathcal{F}_{[\rho^-_k,\rho^+_k]}$ contains exactly one degenerate critical orbit, then the assumption $(3)$ of this theorem is satisfied.
\end{Remark}
Now fix $(v_0,\rho_0)\in(\nabla_v\varphi)^{-1}(0)$ and let 
\[\mathbb{V}^+=\bigoplus_{\lambda\in \sigma^{+}(-\nabla_v^2\varphi(v_0,\rho_0))}\mathbb{V}_{-\nabla^2_v\varphi(v_0,\rho_0)}(\lambda),\]
where $\mathbb{V}_{-\nabla^2_v\varphi(v_0,\rho_0)}(\lambda)$ denotes the eigenspace of $-\nabla^2\varphi(v_0,\rho_0)$ associated to $\lambda.$ 
Notice that a non-degenerate critical $G$-orbit $G(v_0)\subset\Omega$ of the map $\varphi(\cdot,\rho_0)$ is an isolated $\eta_{\rho_0}$-invariant set, where $\eta_{\rho_0}:U\subset\Omega\times\mathbb{R}\rightarrow\mathbb{V}$ is a local $G$-flow induced by the equation 
\[\dot{v}(t)=-\nabla_v\varphi(v(t),\rho_0).\]
For simplicity of notation, $G_{v_0}=H.$ According to Lemma \ref{hesjanpostac}, there is a special form of $\nabla_v^2\varphi(v_0,\rho_0)$ given by the formula (\ref{hesjanpostacwzor}), where \[\mathbb{V}=T_{v_0}G(v_0)\oplus\mathbb{W}^H\oplus(\mathbb{W}^H)^{\bot}=T_{v_0}G(v_0)\oplus\mathbb{V}^+\oplus(\mathbb{V}^+)^{\bot}.\]
Assume that the non-degenerate critical $G$-orbit $G(v_0)$ is special, that is $m^-(C(v_0))=0.$ 
Then the $G$-equivariant Conley index of the critical $G$-orbit $G(v_0)$ has the following $G$-homotopy type:
\[CI_G(G(v_0),\eta_{\rho_0})=([(G\times_H D(\mathbb{V}^+))/(G\times_H S(\mathbb{V}^+))]_G ,[*]),\]
where $D(\mathbb{V}^+),\ S(\mathbb{V}^+)$ denote a closed ball and a sphere in the space $\mathbb{V}^+,$ respectively.
Observe that since $m^-(C(v_0))=0,$ we have $S(\mathbb{V}^+)\subset D(\mathbb{V}^+)\subset\mathbb{V}^+ \subset\mathbb{W}^H$ and $m^-(B(v_0))=\dim\mathbb{V}^+,$ hence
\begin{equation}
CI_G(G(v_0),\eta_{\rho_0})=([(G/H\times D(\mathbb{V}^+))/(G/H\times S(\mathbb{V}^+))]_G ,[*]).
\end{equation}
\begin{Remark}\label{G-CW-complex}
 Additionally, there is a connection between the $G$-equivariant Conley index of the special critical $G$-orbit $G(v_0)$ and a notion of finite, pointed $G$-$CW$-complex. The $G$-equivariant Conley index $CI_G(G(v_0),\eta_{\rho_0})$ has the $G$-homotopy type of a pointed $G$-$CW$-complex which consists of the base point $*$ and one $m^-(B(v_0))$-dimensional cell of orbit type $(H).$
\end{Remark}
\begin{Remark}\label{AchangeofG-CW-complex}
 Thus, if Morse indices of two special critical $G$-orbits are different, then the $G$-equivariant Conley indices have different $G$-homotopy types. It is known (see \cite{[SMOLLER]}) that a change in the $G$-equivariant Conley index implies the existence of $G$-orbits of local bifurcation of critical points of $G$-equivariant variational problems.
\end{Remark} 
 As a consequence of Remark \ref{G-CW-complex} and Remark \ref{AchangeofG-CW-complex}, there are simple conditions verifying the existence of a local bifurcation.  
\begin{Theorem}\label{twbiflok1}
Fix parameters $\rho^{\pm}\!\in\!\mathbb{R},\ \rho^{-}<\rho^{+}$ such that:
\begin{enumerate}
\item[(1)] $\dim \ker \nabla_v ^2\varphi (w(\rho ^{\pm}),\rho ^{\pm}) =  \dim G(w(\rho^{\pm})),$
\item[(2)] $(G_{w(\rho^-)})=(G_{w(\rho^+)}),$
\item[(3)] $m^-(C(w(\rho^{\pm})))=0.$
\end{enumerate}
If $m^{-}(B(w(\rho^-)))\neq m^{-}(B(w(\rho^+))),$ then $(\rho^-,\rho^+)\cap\mathcal{BIF}\neq\emptyset,$ that is there exists a local bifurcation parameter $\rho_0\in(\rho^-\!,\rho^+).$
\end{Theorem}
\begin{proof}
First observe that, according to Remark \ref{G-CW-complex}, the $G$-equivariant Conley indices of the special critical $G$-orbits $G(w(\rho^{\pm}))$ have the $G$-homotopy type of a pointed $G$-$CW$-complex which consists of the base point $*$ and one $m^-(B(w(\rho^\pm)))$-dimensional cell of orbit type $(G_{w(\rho^-)})=(G_{w(\rho^+)}).$ Since  $m^{-}(B(w(\rho^-)))\neq m^{-}(B(w(\rho^+))),$ we conclude, by Remark \ref{AchangeofG-CW-complex}, that $(\rho^-,\rho^+)\cap\mathcal{BIF}\neq\emptyset.$    
\end{proof}
\begin{Theorem}\label{twbiflok2}
Fix parameters $\rho^{\pm}\!\in\!\mathbb{R},\ \rho^{-}<\rho^{+}$ such that:
\begin{enumerate}
\item[(1)] $\dim \ker \nabla_v ^2\varphi (w(\rho ^{\pm}),\rho ^{\pm}) =  \dim G(w(\rho^{\pm})),$
\item[(2)] $(G_{w(\rho^-)})\neq(G_{w(\rho^+)}),$
\item[(3)] $m^-(C(w(\rho^{\pm})))=0.$
\end{enumerate}
Then $(\rho^-,\rho^+)\cap\mathcal{BIF}\neq\emptyset,$ that is there exists a local bifurcation parameter $\rho_0\in(\rho^-\!,\rho^+).$
\end{Theorem}
\begin{proof}
Analogously to the proof of Theorem \ref{twbiflok1}, the $G$-equivariant Conley indices of the special critical $G$-orbits $G(w(\rho^{\pm}))$ have the $G$-homotopy type of a pointed $G$-$CW$-complex which consists of the base point $*$ and one $m^-(B(w(\rho^\pm)))$-dimensional cell of orbit type $(G_{w(\rho^\pm)}),$ see Remark \ref{G-CW-complex}. Since $(G_{w(\rho^-)})\neq (G_{w(\rho^+)}),$ applying Remark \ref{AchangeofG-CW-complex} once again, we get that $(\rho^-,\rho^+)\cap\mathcal{BIF}\neq\emptyset.$     
\end{proof}
\numsec
\section{Applications}
\label{applications} 
In this section we apply our results to show the existence of new families of central configurations, which bifurcate from certain known families of central configurations. We consider $N$ bodies of positive masses $m_1,\ldots ,\ m_N$ in the $2$-dimensional Euclidean vector space, whose positions are denoted by $q_1,\ldots ,q_N\in\mathbb{R}^2$. 
The equations of motion are given by
\[m_j\ddot{q}_j=-\sum_{i=1,i\neq j}^N\mathbf{G}m_jm_i\frac{q_j-q_i}{|q_j-q_i|^3},\ j=1,\ldots, N,\]
where $\mathbf{G}$ is the gravitational constant, which can be set equal to one. 
From now on we will treat the space $\mathbb{R}^{2N}$ as an $SO(2)$-representation $\mathbb{V},$ that is a representation of the Lie group $SO(2)$, which is a direct sum of $N$
copies of the natural, orthogonal $SO(2)$-representation (that is the usual multiplication of a vector by a matrix). The action of $SO(2)$ on $\mathbb{V}$ is given by
\[
SO(2)\times\mathbb{V} \ni (g,q)=(g,(q_1,\ldots ,q_N))\mapsto g\cdot q =(g\cdot q_1,\ldots ,g\cdot q_N)\in\mathbb{V}.
\]
Newtonian equations of motion associated with a potential $U$ have the following form:
\begin{equation} \label{eqr} M\ddot{q}=\nabla_q U(q,m),
\end{equation}
where $M$ is a mass matrix, that is $M=diag (m_1,m_1, \ldots , m_N, m_N),$ and the Newtonian potential $U : \Omega\times(0,+\infty)^N\rightarrow\mathbb{R}$ is given by
\[U(q,m)=U(q_1,\ldots ,q_N,m_1,\ldots ,m_N)=\sum\limits_{1 \leq i<j\leq N}\frac{m_im_j}{|q_i-q_j|},\]
where the set $\Omega\subset\mathbb{V}$ is defined by
\[\Omega=\{q\!=\!(q_1,\ldots ,q_N)\in\mathbb{V} : q_i\neq q_j \ for \ i\neq j\}.\]
We can assume without loss of generality that the centre of mass of the bodies is at the origin of the coordinate chart. Observe that the set
$\Omega\subset \mathbb{V}$ is $SO(2)$-invariant and open, while the potential $U$ is an $SO(2)$-invariant $C^\infty$-map.
\begin{Definition}
A configuration  $q=(q_1,\ldots,q_N)\in\Omega$ is called a central configuration of the system
(\ref{eqr}) if there exists a positive constant $\lambda$ such that $\ddot{q}=-\lambda q.$ 
\end{Definition}
This is equivalent to saying that the following condition is fulfilled:
\begin{equation}\label{eqeq}
-\lambda\nabla_q I(q,m)=\nabla _qU(q,m),
\end{equation} 
where $I:\Omega\times(0,+\infty)^N\rightarrow\mathbb{R}$ given by the formula 
\[
I(q,m)=\frac{1}{2}\sum\limits_{j=1}^{N}m_j|q_j|^2
\]
is the moment of inertia.
Using Euler's theorem and the fact that $U$ is homogeneous of degree $-1$ one can show that $\lambda=\frac{U(q,m)}{2I(q,m)}.$ Moreover,
for any central configuration $q$ it is known that $rq$ and
$g\cdot q$ are central configurations with new $\lambda$ coefficients equal to $\frac{\lambda}{|r|^3}$
and $\lambda$ respectively, where $r\in\mathbb{R}$ and $ g\in SO(2).$ 
Thus this way we can introduce an equivalence relation by saying that two central configurations are equivalent if
we can pass from one to another by a composition of a rotation and a scaling.
However, instead of passing to the quotient and treating classes as single points we shall study the whole $SO(2)$-orbits of central configurations. In this approach our problem becomes a problem of finding critical $SO(2)$-orbits of a smooth, $SO(2)$-invariant potential 
$\hat{\varphi} : \Omega\times (0,+\infty)^N\rightarrow\mathbb{R}$ given by 
\[
 \hat{\varphi}(q,m)=U(q,m)+\lambda I(q,m).
\]
As we mentioned earlier we search bifurcations of central configurations from known families, so we assume additionally that there exist continuous maps $w:(0,+\infty)\rightarrow\Omega$ and $m:(0,+\infty)\rightarrow(0,+\infty)^N$ such that
\[\nabla_q\hat{\varphi}(w(\rho),m(\rho))=0.\]
Now define a potential $\varphi:\Omega\times(0,+\infty)\rightarrow\mathbb{R}$ by 
\[\varphi(q,\rho)=\hat{\varphi}(q,m(\rho)).\]
Then
\[\mathcal{F}=\bigcup_{\rho\in(0,+\infty)}SO(2)(w(\rho ))\times \{\rho \}\subset (\nabla _q \varphi )^{-1}(0).\]
We shall refer to $\mathcal{F}$ as the \emph{trivial family of solutions}.
Thus, we will consider the following equation:
\begin{equation} \label{eqcc}
\nabla_{q}\varphi(q,\rho)=0,
\end{equation}
where $\lambda=\lambda(\rho)=\frac{U(w(\rho),m(\rho))}{2I(w(\rho),m(\rho))}.$
Now we will apply our main Theorems \ref{wkkoniecznybif} and \ref{twierbif1} to formulate necessary and sufficient conditions for the existence of bifurcations of central configurations from our known families $w$. Note that for families $w$ considered in this paper we have $SO(2)_{w(\rho)}=Id$ thus the matrix $C(w(\rho))$ is zero dimensional and in particular, its Morse index is equal to zero.
Now, as has been indicated earlier, we pass from seeking bifurcations of central configurations to seeking bifurcations of critical $SO(2)$-orbits of the $SO(2)$-invariant potential 
\[\varphi(q,\rho)\!=\!U(q,m(\rho))+ \lambda(\rho) I(q,m(\rho)).\]
For this we seek parameters $\rho$ satisfying the necessary condition for the existence of a local bifurcation, which is given by Theorem \ref{wkkoniecznybif}.
Because of tedious computations we aid ourselves by the algebraic processor MAPLE$\texttrademark$ to calculate the Hessian $\nabla^2_q\varphi$ of the potential $\varphi$ and its characteristic polynomial $W_\rho$ along the curve $w.$ Denote the symbolic form of the latter by 
\[W_\rho(q)=q^{2N}-a_1(\rho)q^{2N-1}+\cdots + (-1)^{2N-1}a_{2N-1}(\rho)q+(-1)^{2N} a_{2N}(\rho).\]
Notice that $\dim\ker \nabla^2_q \varphi(w(\rho),\rho)=k$ if and only if 
$a_{2N}=\cdots=a_{2N-k+1}=0$ and $a_{2N-k}\neq 0.$ Additionally in this case $a_{2N-k}$ is the product of the nonzero eigenvalues. Also recall that 
$\dim \ker \nabla_q^2\varphi (w(\rho),\rho)\geq \dim SO(2)(w(\rho))=1,$ so $a_{2N}(\rho)=0.$ For this reason we consider the equation 
\[a_{2N-1}(\rho)=0,\]
whose solutions are the only parameters satisfying the necessary condition from Theorem \ref{wkkoniecznybif}.
In the next step we apply Theorems \ref{twierbif1}, \ref{twchangedegree} and \ref{twbiflok1} to each candidate for the parameter of bifurcation to separate the ones which also satisfy the sufficient conditions.
For each candidate $\rho^*$ we choose its small enough $\varepsilon$-neighbourhood for which we have $\dim \ker \nabla_q ^2\varphi (w(\rho^*\pm\varepsilon),\rho^*\pm\varepsilon) =  \dim SO(2)(w(\rho^*\pm\varepsilon)).$ If
\[ m^{-}(B(w(\rho^*-\varepsilon)))\neq m^{-}(B(w(\rho^*+\varepsilon))),\] then there exists a parameter of local bifurcation in this neighbourhood. Moreover, if these Morse indices differ by an odd number a global bifurcation occurs.
\begin{figure}[h!]
 \centering
  \includegraphics[height=0.4\textwidth]{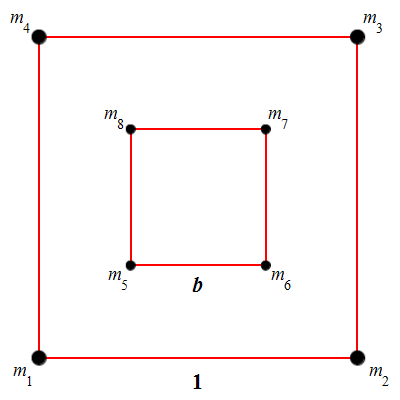}
 \caption{Eight bodies located at the vertices of two nested squares.}\label{two_squares}
 \end{figure}
 
We will now focus on our goal, which is studying bifurcations from the known families of central configurations,  which we will now consider as our trivial family of solutions $w.$ First, we consider nested planar central configuration for the problem of $8$ bodies.
This configuration was found in $2013$ by A. C. Fernandes, L. F. Mello and M. M. da Silva in \cite{[FERNANDES]} and consists of $4$ bodies with equal masses located at the vertices of a square whose side is equal to $1$ and the other $4$ bodies, also with equal masses, located at the vertices of a smaller square whose side is equal to $0<b<0.53177...,$ whose centre coincides with the centre of the first square (see Figure \ref{two_squares}). Denote by $r$ the ratio of the circumcircle of the inner square. We treat $0<r<r_0=0.37602...$ as a parameter. Let $m_1=m_2=m_3=m_4=M_r=-\frac{B(r)}{A(r)}$ and $m_5=m_6=m_7=m_8=1,$ 
where 
\begin{align*}
A(r)&=(R_{1,2}-R_{1,5})\Delta_{1,5,2}+(R_{1,3}-R_{1,6})\Delta_{1,6,3}+(R_{1,7}-R_{1,2})\Delta_{1,7,2},\\
B(r)&=(R_{6,7}-R_{1,5})\Delta_{1,5,6}+(R_{1,7}-R_{6,7})\Delta_{5,6,3}+(R_{1,6}-R_{5,7})\Delta_{1,6,8}
\end{align*}
and $R_{i,j}=1/|q_i-q_j|^3,\ \Delta_{i,j,k}=(q_i-q_j)\wedge (q_i-q_k)$ for $1\leq i,j,k\leq 8.$ 
We define $w:(0,r_0)\rightarrow\Omega$ by
\begin{align}\label{family_two_squares}\displaystyle
w(r)\!\!&=\!\!\left(\!\!-\frac{1}{2},-\frac{1}{2},\frac{1}{2},-\frac{1}{2},\frac{1}{2},\frac{1}{2},-\frac{1}{2},\frac{1}{2}, -\frac{\sqrt{2}}{2}r,-\frac{\sqrt{2}}{2}r, \frac{\sqrt{2}}{2}r,-\frac{\sqrt{2}}{2}r, \frac{\sqrt{2}}{2}r,\frac{\sqrt{2}}{2}r, -\frac{\sqrt{2}}{2}r,\frac{\sqrt{2}}{2}r\!\!\right)\!\!\end{align} 
and $m:(0,r_0)\rightarrow(0,+\infty)^8$ as follows 
\[m(r)=(M_r,M_r,M_r,M_r,1,1,1,1).\] 
\begin{Lemma}\label{lemma1}
Put $ r_1=\frac{1}{7}\sqrt{2}, \ r_2=\frac{1}{6}\sqrt{2}$ and $r_3=\frac{1}{5}\sqrt{2}.$ 
Then $\dim \ker \nabla_q ^2\varphi (w(r_i),r_i) =  \dim SO(2)(w(r_i))=1$ for $i=1,2,3$ and 
the Morse index of $\nabla^2_{q}\varphi(w(\cdot),\cdot)$ evaluated at $r_i$ is 
\begin{align*}
m^- (\nabla^2_{q}\varphi (w(r_i),r_i))=\left\{ \begin{array}{lrl}
1, & for & i=1 \\
3, & for & i=2\\
4, & for & i=3
\end{array}\right. .\end{align*} 
\end{Lemma}
So, there is no matrix $C(w(r))$ and $SO(2)_{w(r)}=Id$ for any $0<r<r_0=0.37602...$ Also, by Lemma \ref{lemma1}, we have $\dim \ker \nabla_q ^2\varphi (w(r_i),r_i) =  \dim SO(2)(w(r_i))=1$ for $i=1,2,3.$ Thus, by Theorem \ref{twbiflok1}, since $m^- (\nabla^2_{q}\varphi (w(r_1),r_1))\neq m^- (\nabla^2_{q}\varphi (w(r_2),r_2)),$ there exists a local bifurcation parameter in the segment $(r_1,r_2).$ Additionally, by Theorem \ref{twierbif1} and Theorem \ref{twchangedegree}, since the numbers $m^- (\nabla^2_{q}\varphi (w(r_2),r_2))$ and $m^- (\nabla^2_{q}\varphi (w(r_3),r_3))$ are of different parity, we have that there exists a global bifurcation parameter in $(r_2,r_3).$

Summarising, we have proved the following theorem:
\begin{Theorem}\label{thmtwosquare}$ $
\begin{enumerate}
 \item There exists a local bifurcation parameter in the segment $(r_1,r_2),$ that is $(r_1,r_2)\cap \mathcal{BIF}\neq\emptyset .$ 
 \item There exists a global bifurcation parameter in the segment $(r_2,r_3),$ that is $(r_2,r_3)\cap \mathcal{GLOB}\neq\emptyset .$ 
\end{enumerate}
\end{Theorem}
Notice that we can consider a subset of the full configuration space $\Omega$ which is invariant for the gradient flow (the set of central configurations of two nested squares type $(r,m,M)$), then studying central configurations in this set becomes a problem of studying zeros of a function $F:(0,r_0)\times(0,+\infty)^2\rightarrow\mathbb{R}$ given by the formula $F(r,m,M)=MA(r)+mB(r).$ For the trivial family of solutions $(r,m,\frac{-mB(r)}{A(r)}),$ for any $(r,m)\in (0,r_0)\times(0,+\infty),$ we have 
$F\left(r,m,M(r,m)\right)=0,$ where $M(r,m)=\frac{-mB(r)}{A(r)},$ and $F'_r\left(r,m,M(r,m)\right)=M(r,m)A'(r)+mB'(r)>0,$ so by the implicit function theorem there is no bifurcation of central configurations of two nested squares type from the trivial family. Combining Theorem \ref{thmtwosquare} and the above, we obtain that the families which bifurcate from the family \eqref{family_two_squares} are not of two nested squares type.
\begin{figure}[h!]
 \centering
  \includegraphics[height=0.4\textwidth]{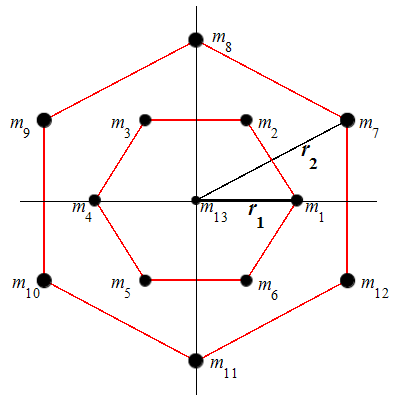}
 \caption{Thirteen bodies in rosette configuration.}\label{rosette}
 \end{figure}

 Now consider a rosette central configuration consisting of $n$ particles of mass $m_1$ lying at the vertices of a regular $n$-gon, $n$ particles of mass $m_2$ lying at the vertices of another $n$-gon, where the second one is rotated by $\frac{\pi}{n},$ and an additional particle of mass $m_0$ lying at the common centre of the two $n$-gons (see \cite{[LEIANDSANTOPRETE]} and \cite{[SEKIGUCHI]}). We analyse bifurcations from this family in the case of $13$ bodies (see Figure \ref{rosette}) which are considered in the following positions:\label{appendixB} 
\begin{center}
 $\hat{q}_{k+1}=\Phi\left(\frac{2\pi}{6}k\right)(r_1,0),\ \hat{q}_{7+k}=\Phi\left(\frac{\pi}{6}+\frac{2\pi}{6}k\right)(r_2,0)$ for $k=0,\ldots ,5$ and $\hat{q}_{13}=(0,0),$  
\end{center}
where by $\Phi$ we denote the universal covering of $SO(2)$ by $\mathbb{R}$, that is a mapping $\Phi : \mathbb{R} \to SO(2)$ given by the formula
\[
 \Phi(s) =\left( \begin{array}{lr} \cos s & -\sin s \\ \sin s & \cos s \end{array}\right).
\]
We transform the coordinates $r_1$ and $r_2$ in the following way: $r_1=r\cos \theta,\ r_2=r\sin\theta,$ where $(r,\theta)\in (0,+\infty)\times(0,\frac{\pi}{2})$ (see \cite{[LEIANDSANTOPRETE]}). 
We choose $\theta=\frac{\pi}{3},\ r=1$ and define $w:(0,+\infty)^2\rightarrow\Omega$ by 
\begin{align}\label{family_rosette}
w(m_0,m_1)=\left(\hat{q}_1,\ldots,\hat{q}_{12},\hat{q}_{13}\right),
\end{align}
where masses $m_0$ and $m_1$ are treated as parameters. 
Notice that for the two-parameter family $w$ we obtain the same results as in Theorem \ref{wkkoniecznybif} and Theorem \ref{twierbif1}. Namely, for any points $(m_0^-,m_1^-)$ and $(m_0^+,m_1^+)$ in the parameter space we can choose a one-parameter path joining these points, that is the segment connecting them.
Then the configuration $\hat{q}\!=\!(\hat{q}_1,\!\ldots\!,\hat{q}_{12},\hat{q}_{13})$ is central for each parameter $(m_0,m_1)\!\in\!(0,+\infty)^2$ if 
\begin{multline*}
m_2=m_2(m_0,m_1)= \frac{1}{ -1862\sqrt {3}-7203+810\sqrt {7}+90\sqrt {7}\sqrt {3} }\left( -7644m_0 \right. \\
\left. + 6\left(81\sqrt {7}-441\sqrt {3}+9\sqrt {3}\sqrt {7}-147 \right)m_1 \right) .
\end{multline*}
So, we define $m:(0,+\infty)^2\rightarrow(0,+\infty)^{13}$ by the formula 
\begin{multline}\nonumber
m(m_0,m_1)=(m_1,m_1,m_1,m_1,m_1,m_1,m_2(m_0,m_1),m_2(m_0,m_1),m_2(m_0,m_1),\\m_2(m_0,m_1),m_2(m_0,m_1),m_2(m_0,m_1),m_0)
\end{multline}
and compute the characteristic polynomial $W_{(m_0,m_1)}$ along the curve $w.$ 
Then, we get that 
\[a_{26}(m_0,m_1)=0\]
and \[a_{25}(m_0,m_1)=C(m_0,m_1)C_1^2(m_0,m_1)C_2(m_0,m_1) C_3^2(m_0,m_1),\]
where polynomials $C_1,\ C_2$ and $C_3$ are equal to $0$ on some curves (see Figure \ref{rosette_zeros}) and $C(m_0,m_1)\neq 0$ for $(m_0,m_1)\in(0,+\infty)^2.$  
\begin{figure}[h!]
 \centering
  \includegraphics[height=0.4\textwidth]{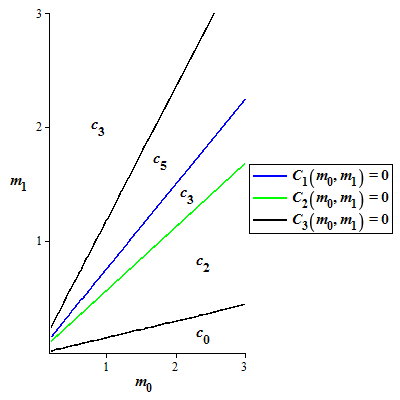}
 \caption{The set of zeros of the coefficient $a_{25}$ and regions $c_3,\ c_5,\ c_2$ and $c_0$ where $a_{25}\neq 0$.}\label{rosette_zeros}
 \end{figure}
\begin{Lemma}\label{lemma2}  
The Morse index of $\nabla^2_{q}\varphi $ along $w$ depends on $(m_0,m_1)$ as follows  
\begin{align*}
m^- (\nabla^2_{q}\varphi (w(m_0,m_1),(m_0,m_1)))=\left\{ \begin{array}{lrl}
3, & for & (m_0,m_1) \in c_3\\
5, & for & (m_0,m_1) \in c_5\\
2, & for & (m_0,m_1) \in c_2\\
0, & for & (m_0,m_1) \in c_0\\
\end{array}\right. 
\end{align*}
and 
\[\dim \ker \nabla_q^2\varphi (w(m_0,m_1),(m_0,m_1))= \dim SO(2)(w(m_0,m_1))=1\] 
for $(m_0,m_1)\in(0,+\infty)^2$ such that $C_i(m_0,m_1)\neq 0$ for all $i=1,2,3.$
\end{Lemma}
So, for any $(m_0,m_1)\in (0,+\infty)^2$ we have that $SO(2)_{w(m_0,m_1)}=Id$ and 
there is no matrix $C(w(m_0,m_1)).$ 
Lemma \ref{lemma2} now shows that $\dim \ker \nabla_q^2\varphi (w(m_0,m_1),(m_0,m_1))=\dim SO(2)(w(m_0,m_1))=1$ for $(m_0,m_1)\!\!\in\!\!(0,+\infty)^2$ such that $C_i(m_0,m_1)\neq 0$ for all $i\in\{1,2,3\}.$ Thus, by Theorem \ref{twbiflok1}, since the numbers $m^- (\nabla^2_{q}\varphi (w(m_0,m_1),(m_0,m_1)))$ are different in the regions $c_3,\ c_5,\ c_2$ and $c_0,$ we have $(m_0,m_1)\in\mathcal{BIF}$ for $(m_0,m_1)\in(0,+\infty)^2$ with $C_i(m_0,m_1)=0,$ where $i\in\{1,2,3\}.$ Moreover, by Theorem \ref{twierbif1} and Theorem \ref{twchangedegree}, since the numbers $m^- (\nabla^2_{q}\varphi (w(m_0,m_1),(m_0,m_1)))$ are of different parity in the regions $c_3$ and $c_2,$ we have $(m_0,m_1)\in\mathcal{GLOB}$ for $(m_0,m_1)\in(0,+\infty)^2$ such that $C_2(m_0,m_1)=0.$

Summarising, we have proved the following theorem:
\begin{Theorem}\label{thmrosette} $ $
\begin{enumerate}
 \item For any $i\in\{1,2,3\},$ if $C_i(m_0,m_1)=0,$ then $(m_0,m_1)\in\mathcal{BIF}.$   
 \item If $C_2(m_0,m_1)=0,$ then $(m_0,m_1)\in\mathcal{GLOB}.$ 
\end{enumerate}
\end{Theorem}
Notice that we can consider a subset of the full configuration space $\Omega$ which is invariant for the gradient flow (the set of central configurations of rosette type $(\frac{r_2}{r_1},m_0,m_1,m_2)$), then studying central configurations in this set becomes a problem of studying zeros of a function $F:(0,+\infty)^3\rightarrow\mathbb{R},\ F(x,\varepsilon,\mu)$ given by the formula $(6)$ from \cite{[LEIANDSANTOPRETE]}, where $x=\frac{r_2}{r_1},\ \varepsilon=\frac{m_2}{m_1}$ and $\mu=\frac{m_0}{m_1}.$ For the trivial family of solutions \eqref{family_rosette}, for any $m_0,\ m_1 \in(0,+\infty),$ we have $F\left(\sqrt{3},\frac{m_2(m_0,m_1)}{m_1},\frac{m_0}{m_1}\right)=0$ and $F'_x\left(\sqrt{3},\frac{m_2(m_0,m_1)}{m_1},\frac{m_0}{m_1}\right)>0,$ so by the implicit function theorem there is no bifurcation of central configurations of rosette type from this family. Combining Theorem \ref{thmrosette} and the above, we obtain that the families which bifurcate from the trivial one are not of 
rosette type.

\begin{figure}[h!]
 \centering
  \includegraphics[height=0.4\textwidth]{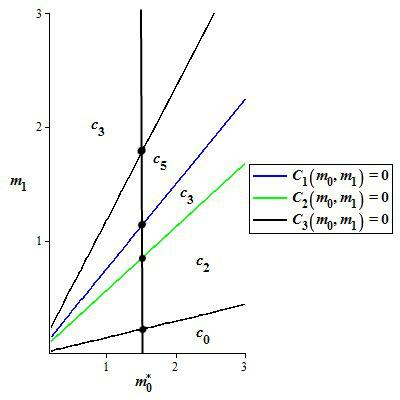}
 \caption{One-parameter family of rosette configuration where $m_1$ is a parameter and the set of four parameters of local bifurcation.}\label{rosette_zeros_oneparameter}
 \end{figure}
Now we fix $m_0^*\in (0,+\infty)$ and treat $w(m_0^*,\cdot )$ and $m(m_0^*,\cdot )$ as one-parameter families where $m_1$ is a parameter. Observe that there are exactly four parameters $m_1$ for which critical orbits $SO(2)(w(m_0^*,m_1))$ are degenerate and there is at most one parameter $m_1^*$ such that $(m_0^*,m_1^*)\in\mathcal{GLOB}$ (see Figure \ref{rosette_zeros_oneparameter}). Then choose $\varepsilon>0$ (small enough) such that $(m_0^*,m_1^*\pm\varepsilon)\notin\mathcal{BIF}$ and there is only one degenerate critical orbit $SO(2)(w(m_0^*,m_1^*))$ in the segment $[m_1^*-\varepsilon,m_1^*+\varepsilon].$ It follows that
\begin{multline*}
BIF_{[m_1^*-\varepsilon,m_1^*+\varepsilon]}\!\!=\!\dg(\nabla_v\varphi(\cdot ,(m_0^*, m_1^*+\varepsilon)),\Theta^+)- \dg(\nabla_v\varphi(\cdot, (m_0^* ,m_1^*-\varepsilon)) ,\Theta^-)=\\
=\!(-1)^{m^-(B(w(m_0^*,m_1^*+\varepsilon)))}\!\chi_{SO(2)}(SO(2)/\{Id\}^+)-(-1)^{m^-(B(w(m_0^*,m_1^*-\varepsilon)))}\!\chi_{SO(2)}(SO(2)/\{Id\}^+)\\
=((-1)^{m^-(B(w(m_0^*,m_1^*+\varepsilon)))}-(-1)^{m^-(B(w(m_0^*,m_1^*-\varepsilon)))})\chi_{SO(2)}(SO(2)/\{Id\}^+)=\\
=((-1)^{3}-(-1)^{2})\chi_{SO(2)}(SO(2)/\{Id\}^+)=-2\chi_{SO(2)}(SO(2)/\{Id\}^+)\neq\mathbf{0}\in U(G). 
\end{multline*}
Similarly, for any $m_1\neq m_1^*$ since the numbers $m^-(B(w(m_0^*,m_1\pm\tilde{\varepsilon})))$ are of the same parity, we conclude that $BIF_{[m_1-\tilde{\varepsilon},m_1+\tilde{\varepsilon}]}=\mathbf{0},$ where $\tilde{\varepsilon}$ is chosen as above.   
Therefore, by Theorem \ref{twierbif2}, $C([m_1^*-\varepsilon,m_1^*+\varepsilon])$ is not compact.

\end{document}